\documentclass[twoside,11pt]{article}
\usepackage{a4}
\usepackage{amssymb}
\usepackage{amsmath}
\usepackage{amsthm}
\usepackage{enumerate}
\usepackage{color}

\newtheorem{theorem}{Theorem}
\newtheorem{corollary}{Corollary}
\newtheorem{lemma}{Lemma}

\newtheorem*{claim*}{Claim}

\newtheorem*{conjecture*}{Conjecture}

\newtheorem{theorem MTP}{Mass Transference Principle}
\newtheorem*{theorem MTP*}{Mass Transference Principle}
\newtheorem{theorem GMTP}{Theorem MTP*}
\newtheorem*{theorem GMTP*}{Theorem MTP*}
\newtheorem{theorem K}{Khintchine's Theorem}
\newtheorem*{theorem K*}{Khintchine's Theorem}
\newtheorem{theorem J}{Jarn\'{\i}k's Theorem}
\newtheorem*{theorem J*}{Jarn\'{\i}k's Theorem}
\newtheorem{theorem KJ}{Khintchine--Jarn\'{\i}k Theorem}
\newtheorem*{theorem KJ*}{Khintchine--Jarn\'{\i}k Theorem}
\newtheorem{theorem BV1}{Theorem BV1}
\newtheorem*{theorem BV1*}{Theorem BV1}
\newtheorem{theorem BV2}{Theorem BV2}
\newtheorem*{theorem BV2*}{Theorem BV2}
\newtheorem{theorem KG}{Theorem KG}
\newtheorem*{theorem KG*}{Theorem KG}
\newtheorem{theorem IHKG}{Inhomogeneous Khintchine--Groshev Theorem}
\newtheorem*{theorem IHKG*}{Inhomogeneous Khintchine--Groshev Theorem}
\newtheorem{theorem DLN1}{Theorem DLN1}
\newtheorem*{theorem DLN1*}{Theorem DLN1}
\newtheorem{theorem DLN2}{Theorem DLN2}
\newtheorem*{theorem DLN2*}{Theorem DLN2}
\newtheorem{theorem DLN3}{Theorem DLN3}
\newtheorem*{theorem DLN3*}{Theorem DLN3}
\newtheorem{theorem S}{Theorem S}
\newtheorem*{theorem S*}{Theorem S}
\newtheorem{theorem AB}{Theorem AB}
\newtheorem*{theorem AB*}{Theorem AB}
\theoremstyle{definition}

\theoremstyle{remark}

\newtheorem*{remark*}{Remark}
\newtheorem{example}{Example}
\newtheorem*{example*}{Example}

\renewcommand{\Bbb}[1]{\mathbb{#1}}
\newcommand{\N}{{\Bbb N}}         

\newcommand{\R}{{\Bbb R}}        
\newcommand{\K}{{\Bbb K}}         

\newcommand{\cA}{{\cal A}}

\newcommand{\cC}{{\cal C}}

\newcommand{\cF}{{\cal F}}
\newcommand{\cG}{{\cal G}}
\newcommand{\cH}{{\cal H}}

\newcommand{\cL}{{\cal L}}

\newcommand{\cR}{{\cal R}}

\newcommand{\Om}{\Omega}
\newcommand{\U}{\Upsilon}
\newcommand{\La}{\Lambda}
\newcommand{\ka}{\kappa}

\newcommand{\tU}{\tilde\Upsilon}

\newcommand{\Pj}{\Phi_j}

\newcommand{\diam}{r}
\newcommand{\dist}{\operatorname{dist}}

\newcommand{\vv}[1]{{\mathbf{#1}}}

\renewcommand{\le}{\leq}
\renewcommand{\ge}{\geq}

\newcommand{\kgb}{K_{G,B}}

\newcommand{\hf}{\cH^f}
\newcommand{\hs}{\cH^s}

\newcommand{\Lj}{(L;j)}
\newcommand{\aj}{(A;j)}
\newcommand{\ajj}{(A';j')}

\newcommand{\caj}{\cC\aj}

\DeclareMathOperator{\dimh}{dim_H}
\DeclareMathOperator{\udimb}{\overline{dim}_B}
\DeclareMathOperator{\ldimb}{\underline{dim}_B}
\DeclareMathOperator{\dimb}{dim_B}

\newcommand{\Addresses}{{
  \bigskip
  \footnotesize

  D.~Allen, \textsc{School of Mathematics, University of Manchester, Oxford Road, Manchester, M13 9PL, UK}\par\nopagebreak
  \textit{E-mail address:} \texttt{demi.allen@manchester.ac.uk}

  \medskip

  S.~Baker, \textsc{Mathematics Institute, Zeeman Building, University of Warwick, Coventry, CV4 7AL, UK}\par\nopagebreak
  \textit{E-mail address:} \texttt{simonbaker412@gmail.com}
}}

\begin{document}

\title{\bf A General Mass Transference Principle}
\author{Demi Allen\footnote{EPSRC Doctoral Prize Fellow supported by grant EP/N509565/1} \and Simon Baker\footnote{Supported by EPSRC grant EP/M001903/1}}
\date{\footnotesize{\it In loving memory of Mandy Jayne Allen  (1968--2018).}}
\maketitle
\begin{abstract}
The Mass Transference Principle proved by Beresnevich and Velani in 2006 is a celebrated and highly influential result which allows us to infer Hausdorff measure statements for $\limsup$ sets of balls in $\R^n$ from \emph{a priori} weaker Lebesgue measure statements. The Mass Transference Principle and subsequent generalisations have had a profound impact on several areas of mathematics, especially Diophantine Approximation. In the present paper, we prove a considerably more general form of the Mass Transference Principle which extends known results of this type in several distinct directions. In particular, we establish a mass transference principle for $\limsup$ sets defined via neighbourhoods of sets satisfying a certain local scaling property. Such sets include self-similar sets satisfying the open set condition and smooth compact manifolds embedded in $\mathbb{R}^n$. Furthermore, our main result is applicable in locally compact metric spaces and allows one to transfer Hausdorff $g$-measure statements to Hausdorff $f$-measure statements. We conclude the paper with an application of our mass transference principle to a general class of random $\limsup$ sets.  
\end{abstract}
\bigskip


\section{Introduction}\label{sec:1}

\subsection{Background} \label{background section}
In Diophantine Approximation, Dynamical Systems, and Probability Theory, many sets of interest can be characterised as $\limsup$ sets. Recall that, given a countable collection of sets $(E_j)_{j \in \N},$ we define the corresponding $\limsup$ set to be
\begin{align*}
\limsup_{j \to \infty}{E_j} &:=\bigcap_{j=1}^{\infty}\bigcup_{n=j}^{\infty}E_j\\
                            &=\big\{x:x\in E_j \textrm{ for infinitely many }j\in\mathbb{N}\big\}.
\end{align*} Often we are interested in determining the metric properties of $\limsup E_j.$ When the sequence $(E_j)_{j \in \N}$ is a collection of balls a powerful tool in determining the metric properties of $\limsup E_j$ is the Mass Transference Principle~\cite{BV MTP}, which allows us to infer Hausdorff measure statements from seemingly less general Lebesgue measure statements.

Given a ball $B:=B(x,r)$ in $\R^n$ and a dimension function $f:\R^+ \to \R^+$ (see Section \ref{preliminaries section} for definitions), define another corresponding ball $B^f:=B(x,f(r)^{\frac{1}{n}})$. Throughout, $\R^+:=[0,\infty)$. When $f(r)=r^s$ for some real number $s > 0$, we write $B^s$ in place of $B^f$. In particular, $B^n = B$ for $n \in \N$. The following was established by Beresnevich and Velani in \cite{BV MTP}.

\begin{theorem MTP*} \label{mtp theorem}
Let $(B_j)_{j\in\N}$ be a sequence of balls in $\R^n$ with
$\diam(B_j)\to 0$ as $j\to\infty$. Let $f$ be a dimension function
such that $x^{-n}f(x)$ is monotonic and suppose that, for any ball $B$ in $\R^n$,
\[ \cH^n\big(\/B\cap\limsup_{j\to\infty}B^f_j{}\,\big)=\cH^n(B) \ .\]
Then, for any ball $B$ in $\R^n$,
\[ \cH^f\big(\/B\cap\limsup_{j\to\infty}B^n_j\,\big)=\cH^f(B) \ .\]
\end{theorem MTP*} 

We denote by $\cH^f(X)$ the Hausdorff $f$-measure of a set $X \subset \R^n$. For $s \geq 0$, $\cH^s(X)$ denotes the standard Hausdorff $s$-measure of $X$. These notions will be formally introduced in Section \ref{preliminaries section}. It is worth noting at this point though that if $X$ is a Borel subset of $\R^n$, then $\cH^n(X)$ is a constant multiple times the $n$-dimensional Lebesgue measure of $X$ (see \cite{Falconer ref} for further details). Thus, as discussed previously, the Mass Transference Principle genuinely does enable us to transfer Lebesgue measure statements to Hausdorff measure ones.

A generalisation of the Mass Transference Principle, which is applicable to $\limsup$ sets of balls in locally compact metric spaces and not restricted to $\limsup$ sets in $\R^n$, was also given by Beresnevich and Velani in \cite{BV MTP}. Furthermore, this generalisation allows for the transference of Hausdorff $g$-measure (not just Lebesgue measure) statements to Hausdorff $f$-measure statements, where $g$ and $f$ are dimension functions subject to some mild conditions. Before stating this result formally, we require some preliminaries.

We say that a function $f: \R^+ \to \R^+$ is \emph{doubling} if there exists a constant $\lambda>1$ such that 
\begin{equation}
\label{doubling}
f(2x)<\lambda f(x)
\end{equation} for all $x>0$. 

Let $(X,d)$ be a locally compact metric space and let $g$ be a doubling dimension function. Moreover, suppose that there exist constants \mbox{$0<c_1<1<c_2<\infty$} and $r_0>0$ such that 
\begin{equation}
\label{gball}
c_1 g(r)\leq \cH^g(B(x,r))\leq c_2 g(r)
\end{equation} for all $x\in X$ and $0<r<r_0$. Given another dimension function $f$ and a ball $B:=B(x,r)$ in $X$ we define $$B^{f,g}:=B(x,g^{-1}(f(r))).$$ The following theorem was established in \cite{BV MTP}.

\begin{theorem GMTP*}
Let $(X,d)$ and $g$ be as above and let $(B_j)_{j \in \N}$ be a sequence of balls in $X$ with $r(B_j)\to 0$ as $j\to \infty$. Let $f$ be a dimension function such that $f/g$ is monotonic and suppose that for any ball $B$ in $X$ we have 
$$\mathcal{H}^g(B\cap \limsup_{j \to \infty} B_j^{f,g})=\mathcal{H}^g(B).$$ 
Then, for any ball $B$ in $X$,
$$\mathcal{H}^f(B\cap \limsup_{j \to \infty} B_j)=\mathcal{H}^f(B).$$ 
\end{theorem GMTP*}

Compared with the Mass Transference Principle, this general theorem applies to $\limsup$ sets of balls in more general metric spaces and deals with a larger class of (Hausdorff) measures. As an example, Theorem MTP* is applicable when $X$ is, say, the middle-third Cantor set. In this case, Theorem MTP* has been utilised in \cite{LSV ref} to solve a problem posed by Mahler regarding the existence of very well approximable points in the middle-third Cantor set.

The Mass Transference Principle was originally motivated by a desire to establish a Hausdorff measure analogue of the famous Duffin--Schaeffer Conjecture in Metric Number Theory. Since their initial announcement, the Mass Transference Principle and Theorem MTP* have been shown to have applications in many distinct areas of mathematics. In particular, the fields of Number Theory, Dynamical Systems, and Fractal Geometry have all benefited significantly from these results. For further applications of the Mass Transference Principle and Theorem MTP* see \cite{BBDV ref, BRV aspects, BV MTP, KRW, Sun}. 


In the Euclidean setting, Beresnevich and Velani extended the Mass Transference Principle in another direction to allow for the transference of Lebesgue measure statements to Hausdorff measure statements for $\limsup$ sets defined via neighbourhoods of ``approximating planes'' \cite{BV Slicing}. The result they obtained in this case, \cite[Theorem 3]{BV Slicing}, was subject to extra technical conditions arising from their particular proof strategy. Recently, the first author and Beresnevich have removed these additional constraints in \cite{AB ref}, thus unlocking a number of previously inaccessible applications in Diophantine Approximation. We include here a statement of \cite[Theorem 1]{AB ref} to enable ease of comparison of this mass transference principle for planes with the statement of the main result of the present article, namely Theorem \ref{general mtp theorem} below.

Let $n, m \geq 1$ and $l \geq 0$ be integers such that $n = m + l$. Let $\cR := (R_j)_{j \in \N}$ be a family of $l$-dimensional planes in $\R^n$. For every $j \in \N$ and $\delta \geq 0$, define
\[ \Delta(R_j,\delta) := \{\mathbf{x} \in \R^n : \dist(\mathbf{x},R_j)<\delta\},\]
where $\dist(\vv x,R_j)=\inf\{\|\vv x-\vv y\|:\vv y\in R_j\}$ and $\|\cdot\|$ is any fixed norm on $\R^n$.

Let $\U:= (\U_j)_{j \in \N}$ be a sequence of non-negative reals such that $\U_j \to 0$ as $j \to \infty$.
Consider
\[ \La(\U) := \{\mathbf{x} \in \R^n : \mathbf{x} \in \Delta(R_j, \U_j) \text{ for infinitely many } j \in \N\}.\]
The following is shown in \cite{AB ref}.

\begin{theorem AB*} \label{mtp for linear forms theorem}
Let $\cR$ and $\U$ be as given above. Let $f$ and $g : r \to g(r) := r^{-l}f(r)$ be dimension functions such that $r^{-n}f(r)$ is monotonic and let $\Om$ be a ball in $\R^n$. Suppose that, for any ball $B$ in $\Om$,
\[\cH^n\left(B \cap \La\left(g(\U)^{\frac{1}{m}}\right)\right) = \cH^n(B).\]
Then, for any ball $B$ in $\Om$,
\[\cH^f(B \cap \La(\U)) = \cH^f(B).\]
\end{theorem AB*}

While we will be concerned here with generalising the aforementioned variations of the Mass Transference Principle, we remark here, for completeness, that progress towards mass transference principles has also been made in some other settings. For example, progress towards proving a mass transference principle for rectangles in the Euclidean setting has been made by Wang, Wu and Xu in~\cite{WWX ref} and an implicit multifractal mass transference principle is given by Fan, Schmeling and Troubetzkoy in \cite{FST ref}.

In this paper, we extend the results of \cite{AB ref, BV MTP, BV Slicing} by proving a general version of the Mass Transference Principle that applies to $\limsup$ sets in a locally compact metric space which are defined in terms of neighbourhoods of sets satisfying a certain local scaling property (see Section \ref{main result section}). Our main result, Theorem \ref{general mtp theorem}, extends the known mass transference principles of \cite{AB ref, BV MTP, BV Slicing} in several manners. First of all, while it incorporates mass transference principles for balls and planes, Theorem \ref{general mtp theorem} is also applicable to more exotic sets. For example, we are able to consider $\limsup$ sets generated by sequences of neighbourhoods of smooth compact manifolds or self-similar fractals satisfying the open set condition. Furthermore, unlike in previously known variants of the Mass Transference Principle, as long as the local scaling property is satisfied, the sets generating the $\limsup$ sets in Theorem \ref{general mtp theorem} need not all be of the same type (e.g. all balls or all planes). Secondly, we deal with $\limsup$ sets in a locally compact metric space $(X,d)$ and are not confined to the Euclidean setting. Finally, the result we derive allows us to transfer Hausdorff $g$-measure statements to Hausdorff $f$-measure statements, where $f$ and $g$ are dimension functions subject to some mild conditions. 

Compared with its predecessors, the greater generality of Theorem \ref{general mtp theorem} opens up a number of new possible applications to explore. We include one such application in Section~\ref{Application} where we use Theorem~\ref{general mtp theorem} to deduce Hausdorff measure and dimension results for a family of random $\limsup$ sets. Within Diophantine Approximation, it is reasonable to expect that Theorem \ref{general mtp theorem} will enable the establishment of further Hausdorff measure statements relating to approximation on manifolds (see \cite{BDV limsup sets, BRV aspects} and the references therein for more on this problem). Within Dynamical Systems, it is also reasonable to expect that Theorem \ref{general mtp theorem} will allow one to study a wider class of shrinking target problems, in particular when our target is allowed to have a more exotic structure (see \cite{PR, SW, Sun} and the references therein for more on this problem). We hope to return to these topics in a later work. 


\subsection{The Main Result} \label{main result section}
Let $(X,d)$ be a locally compact metric space and let $g$ be a doubling dimension function satisfying \eqref{gball}. Given $F\subset X$ and $\delta \geq 0$, we define the $\delta$-neighbourhood of $F$ to be
$$\Delta(F,\delta):=\{x\in X:d(x,F)<\delta\},$$
where $d(x,F):=\min\{d(x,y): y \in F\}$.
 
The following \emph{local scaling property} appears to be the key which enables us to prove a ``unifying'' mass transference principle which incorporates and extends the results presented in Section \ref{background section}.

\noindent{\bf Local Scaling Property (LSP):}  
Given a sequence of sets \mbox{$\cF:=(F_j)_{j \in \N}$} in $X$ and $0\leq \kappa< 1$, we say that $\cF$ satisfies the \emph{local scaling property (LSP) with respect to $\ka$} if there exist constants $c_3,c_4,r_1>0$ such that, for any $0<r<r_1$, $\delta<r$, $j \in \N$ and $x\in F_j,$ we have 
\begin{equation} \label{local scaling}
c_3 g(\delta)^{1-\ka}\cdot g(r)^{\ka}\leq \cH^g(B(x,r)\cap \Delta(F_j,\delta))\leq c_4  g(\delta)^{1-\ka}\cdot g(r)^{\ka}.
\end{equation}

If \eqref{local scaling} is satisfied for one specific set $F$, say, we will also say that $F$ satisfies the local scaling property with respect to $\ka$. It should be clear from context when we are referring to an individual set and when we are referring to a sequence of sets.

If we restrict ourselves to the case where $X=\mathbb{R}^n$ and $\cH^g=\cH^n$ (that is, essentially, Lebesgue measure), then \eqref{local scaling} takes the form
\begin{equation} \label{Euclidean local scaling}
c_3 \delta^{n-\ka n}\cdot r^{\ka n}\leq \cH^n(B(x,r)\cap \Delta(F_j,\delta))\leq c_4  \delta^{n-\ka n}\cdot r^{\ka n}.
\end{equation}

Many familiar subsets of Euclidean space satisfy \eqref{Euclidean local scaling} for an appropriate choice of $\kappa$. For example, in the next section we show that any smooth compact manifold embedded in $\mathbb{R}^n$ satisfies \eqref{Euclidean local scaling}. In addition, we show that self-similar sets satisfying the open set condition also satisfy the LSP. In these examples we will see that $\kappa$ is related to the box counting dimension of the sets we are considering and we offer some discussion as to why this should be the case. In the meantime, we present here two trivial examples of sets satisfying the~LSP.
\begin{example} \label{points example}
If $\cF$ were a sequence of points, then it follows from \eqref{gball} that $\cF$ satisfies the LSP with respect to $\kappa=0$. 
\end{example}
\begin{example}
Suppose $F\subset \mathbb{R}^n$ is a specific set satisfying \eqref{Euclidean local scaling}. One can then define a sequence $\cF := (F_j)_{j \in \N}$ by defining each $F_j$ to be the image of $F$ under some isometry. Clearly $\cF$ provides us with a sequence of sets satisfying the LSP.
\end{example}

Next, suppose $(\U_j)_{j\in\N}$ is a sequence of non-negative reals such that $\U_j \to 0$ as $j \to \infty.$ Consider the $\limsup$ set
\[\La(\U):=\{x \in X: x \in \Delta(F_j,\U_j) \text{ for infinitely many } j \in \N\}.\]

Our main result is the following theorem.

\vbox{
\begin{theorem} \label{general mtp theorem}
Let $(X,d)$ and $g$ be as above. Let $\cF:=(F_j)_{j \in \N}$ be a sequence of sets satisfying the LSP with respect to some $0 \leq \ka < 1,$ and let $\U := (\U_j)_{j \in \N}$ be a sequence of non-negative reals such that $\U_j \to 0$ as $j \to \infty$. Let $f$ be a dimension function such that $f/g$ is monotonic and $f/g^{\ka}$ is a dimension function. Suppose that for any ball $B$ in $X$ we have 
	\begin{align} \label{full measure assumption}
	\cH^g\left(B \cap \La\left(g^{-1}\left(\left(\frac{f(\U)}{g(\U)^{\ka}}\right)^{\frac{1}{1-\ka}}\right)\right)\right)=\cH^g(B).
	\end{align}
	Then, for any ball $B$ in $X$ we have 
	\[\cH^f(B \cap \La(\U)) = \cH^f(B).\]
\end{theorem}}

Taking $(F_j)_{j \in \N}$ to be a sequence of points in $X$ and $\kappa = 0$, as in Example~\ref{points example}, Theorem~\ref{general mtp theorem} coincides with Theorem MTP* given above. If, further, we insist that $X = \R^n$ and $g(r) = r^n$, we recover the usual Mass Transference Principle. Theorem AB can also be deduced as a special case of Theorem~\ref{general mtp theorem} by taking $X = \Omega$ to be a suitable ball in $\R^n$, $(F_j)_{j \in \N}$ to be a sequence of $l$-dimensional planes $(R_j)_{j \in \N}$ in $\R^n$, $\kappa = \frac{l}{n}$, and $g(r)=r^n$. 

Restricting our attention to Euclidean space and $s$-dimensional Hausdorff measures, Theorem~\ref{general mtp theorem} takes the following simpler form. 
\begin{corollary} \label{general euclidean mtp theorem}
	Let $\cF:=(F_j)_{j \in \N}$ be a sequence of subsets of $\mathbb{R}^n$ satisfying the LSP with respect to some $0\leq \kappa <1$ and let $\U := (\U_j)_{j \in \N}$ be a sequence of non-negative reals such that $\U_j \to 0$ as $j \to \infty$. Let $s>\kappa n$ and suppose that for any ball $B$ in $\mathbb{R}^n$ we have 
	\begin{align*}
	\cH^n(B\cap \La(\U^{\frac{s-\kappa n}{(1-\kappa)n}}) )=\cH^n(B).
	\end{align*}
	Then, for any ball $B$ in $X$ we have 
	\[\cH^s(B \cap \La(\U)) = \cH^s(B).\]
\end{corollary}

\subsection{Structure of the paper}
The rest of the paper is arranged as follows. Some examples of sets satisfying the LSP are outlined in Section \ref{lsp examples}. In Section \ref{preliminaries section} we recall some useful preliminaries from geometric measure theory. In Section \ref{KGB Section} we prove some key technical lemmas which will be required throughout the proof of Theorem~\ref{general mtp theorem}, which is presented in full in Section \ref{proof section}. Finally, in Section \ref{Application}, we use Theorem~\ref{general mtp theorem} to deduce the Hausdorff dimension and measure of some random $\limsup$ sets.


\section{Some sets satisfying the local scaling property}\label{lsp examples}

In this section we provide some discussion on connections between the LSP and box counting dimension and Minkowski content. Thereafter, we briefly detail two collections of sets that satisfy the LSP for appropriate choices of parameters.

\subsection{Box counting dimension, Minkowski content and the LSP} \label{box dimension}

In this section, we highlight the connection between \eqref{Euclidean local scaling} and the Minkowski/box counting dimension of a set and why these quantities are likely to be related to $\ka$ for sets satisfying the LSP.

Instead of the usual box counting definition of lower and upper box counting dimension (see, for example, \cite{Falconer ref}), there is the following equivalent notion that is defined using the volume of a $\delta$-neighbourhood of a set. Given a bounded set $F\subset \mathbb{R}^n$, the \emph{lower} and \emph{upper box counting dimension of $F$} are defined to be, respectively, 
$$\ldimb(F):=n-\limsup_{\delta\to 0} \frac{\log \mathcal{H}^{n}(\Delta(F,\delta))}{\log \delta}$$ 
and 
$$\udimb(F):=n-\liminf_{\delta\to 0} \frac{\log \mathcal{H}^{n}(\Delta(F,\delta))}{\log \delta}.$$ 
When these limits coincide we call the common value the \emph{box counting dimension of $F$} and denote it by $\dimb(F)$. It follows from this definition that if~$F$ were a set whose box counting dimension exists, and we were interested in whether~\eqref{Euclidean local scaling} held for some $\kappa$, then the natural candidate for $\kappa$ would be $\dim_{B}(F)\cdot n^{-1}.$

Two other useful quantities that describe how the volume of a $\delta$-neighbourhood of a set $F$ scales are the lower and upper Minkowski content. Given a bounded set $F$ contained in $\mathbb{R}^n$ whose box counting dimension exists, we define the \emph{lower} and \emph{upper Minkowski content} to be, respectively,
$$\underline{M}(F):=\liminf_{\delta \to 0} \delta^{n-\dimb(F)}\cdot\cL^n(\Delta(F,\delta))$$ and $$\overline{M}(F):=\limsup_{\delta \to 0} \delta^{n-\dimb(F)}\cdot\cL^n(\Delta(F,\delta)).$$
Here $\cL^n$ is the $n$-dimensional Lebesgue measure. When $\underline{M}(F)=\overline{M}(F)$ we call the common value the \emph{Minkowski content} and denote it by $M(F).$ Determining conditions under which a set $F$ has both $\underline{M}(F)$ and $\overline{M}(F)$  positive and finite is a well studied problem, see \cite{Kombrink ref}. Equation \eqref{Euclidean local scaling} does not follow directly from the positivity and finiteness of both $\underline{M}(F)$ and $\overline{M}(F).$ However, if $F$ were reasonably homogeneous in small neighbourhoods, we would expect this property to be sufficient to deduce \eqref{Euclidean local scaling}. 

\subsection{Smooth compact manifolds}
Let $M\subset \mathbb{R}^n$ be a smooth compact manifold of dimension $l$. Let us start by remarking that when $M$ is a smooth compact manifold, the tangent space map sending $x\to T_x M$ is a continuous map from $M$ into the Grassmanian of $l$-dimensional subspaces of $\mathbb{R}^n$. Fixing $x\in M$, and applying a rotation if necessary, we can identify $T_x M$ with $\mathbb{R}^l \times \{0^{n-l}\}.$ Since the tangent space map is continuous we know that for any $y$ sufficiently close to $x$ its tangent space $T_y M$ is approximately $\mathbb{R}^l \times \{0^{n-l}\}.$ As such, we can assert that there exists $R_x>0$ such that the following two properties hold:
\begin{itemize}
	\item For any $y\in B(x,R_x)\cap M$ and $r>0$ such that $B(y,r)\subset B(x,R_x),$ we have 
	\begin{equation}
	\label{projection volume}
	\cH^l(\pi_l(B(y,r)\cap M))\asymp r^l.
	\end{equation}Here $\pi_l$ is the projection map from $\mathbb{R}^n$ to $\mathbb{R}^l$ sending $(x_1,\ldots,x_n)$ to $(x_1,\ldots,x_l)$.
	\item For any $y\in B(x,R_x)\cap M$ and $\delta<R_x$ we have 
	\begin{equation}
	\label{fibre volume}
	\cH^{n-l}(\{z\in\Delta(M,\delta):\pi(z)=\pi(y)\})\asymp \delta^{n-l}.
	\end{equation}
\end{itemize}The implied constants in \eqref{projection volume} and \eqref{fibre volume} depend only upon  $M$ and $x$. Via an application of \eqref{projection volume}, \eqref{fibre volume}, and Fubini's Theorem, it can be shown that for any $y\in B(x,R_x/2),$ $r\leq R_x/2,$ and $\delta\leq r$ we have 
\begin{equation}
\label{local scaling property at x}
\cH^n(B(y,r)\cap \Delta(M,\delta))\asymp r^l \delta^{n-l}.
\end{equation} 
Again, the implied constants in \eqref{local scaling property at x} depend only upon $M$ and $x$. Since $\{B(x,R_x/2)\}_{x\in M}$ covers $M,$ it follows by a compactness argument that there exists $R>0$ such that for any $y\in M,$ $r\leq R$ and $\delta\leq r$ we have $$\cH^n(B(y,r)\cap \Delta(M,\delta))\asymp r^l \delta^{n-l}.$$ 
Thus, we conclude that $M$ satisfies the LSP with respect to $\ka = \frac{l}{n}$.

\subsection{Self-similar sets satisfying the open set condition}
Let $\Phi=\{\phi_1,\ldots,\phi_l\}$ be a collection of contracting similarities acting on $\R^n$; that is, $\Phi$ is a collection of maps such that
\[|\phi_i(x)-\phi_i(y)| = r_i|x-y| \quad \text{for all } x,y \in \R^n,\]
and $0 < r_i < 1$ for each $1 \leq i \leq l$. It is well-known (see, for example, \cite{Falconer ref}) that there exists a unique non-empty compact set $K\subset \mathbb{R}^n$ such that 
$$K=\bigcup_{i=1}^l\phi_i(K).$$ 

We say that $\Phi$ satisfies the \emph{open set condition} if there exists an open set $O\subset\mathbb{R}^n$ such that $\phi_i(O)\subset O$ for each $1 \leq i \leq l$ and $\phi_i(O)\cap \phi_j(O)=\emptyset$ whenever $i\neq j$. In \cite{Gat} it is shown that if $\Phi$ satisfies the open set condition then there exist constants $b_1,b_2>0$ such that for all $\delta$ sufficiently small
\begin{equation}
\label{self-similar scaling}
b_1 \delta^{n-d}\leq \cH^n(\Delta(K,\delta))\leq b_2 \delta^{n-d},
\end{equation}
where $d$ is the box counting dimension of $K$. In \cite{Gat}, \eqref{self-similar scaling} was proved for the $n$-dimensional Lebesgue measure. The statement given above follows since $\cH^n$ is equal to a scalar multiple of the $n$-dimensional Lebesgue measure. 

We now prove that the local scaling property holds for $K$ with $\ka=\frac{d}{n}$. To this end, fix $x\in K$ and some small number $r > 0$. Then there exists a sequence $(a_i)_{i \in \N}\in\{1,\ldots,l\}^{\mathbb{N}}$ such that 
$$\bigcap_{m=1}^{\infty} (\phi_{a_1}\circ \cdots \circ \phi_{a_m})(K)=x.$$ Next, note that there exists $N\in\mathbb{N}$ such that 
$$(r_{a_1}\cdots r_{a_{N-1}})\cdot Diam K\geq r/2$$ 
and 
$$(r_{a_1}\cdots r_{a_N})\cdot Diam K< r/2.$$ 
For this value of $N$ and $\delta<r/2$, we have 
\begin{align} \label{self-similar LSP lower bound}
\cH^n(B(x,r)\cap \Delta(K,\delta))&\geq \cH^n(B(x,r)\cap \Delta((\phi_{a_1}\circ \cdots \circ \phi_{a_{N}})(K),\delta)) \nonumber \\
&\geq \cH^n(\Delta((\phi_{a_1}\circ \cdots \circ \phi_{a_{N}})(K),\delta)) \nonumber \\
&=\cH^n((\phi_{a_1}\circ \cdots \circ \phi_{a_{N}})(\Delta(K,\delta\cdot(r_{a_1}\cdots r_{a_N})^{-1}))) \nonumber \\
&=(r_{a_1}\cdots r_{a_N})^n\cH^n(\Delta(K,\delta\cdot(r_{a_1}\cdots r_{a_N})^{-1})) \nonumber \\
&\stackrel{\eqref{self-similar scaling}}\geq b_1 \cdot (r_{a_1}\cdots r_{a_N})^n \cdot \left(\frac{\delta}{r_{a_1}\cdots r_{a_N}}\right)^{n-d} \nonumber \\
&= b_1 \cdot (r_{a_1}\cdots r_{a_N})^d\cdot \delta^{n-d} \nonumber \\
&\asymp b_1 r^{d} \delta^{n-d}.
\end{align} 
The last line follows since $r_{a_1}\cdots r_{a_N} \asymp r$ by our choice of $N$. Thus, we have proved that the lower bound in the LSP holds in this case. It remains to prove the upper bound.

Given $r>0$ let 
$$I_r:=\{(a_1,\ldots,a_k)\in\cup_{j=0}^{\infty}\{1,\ldots, l\}^j:r_{a_1}\cdots r_{a_k}\leq r <r_{a_1}\cdots r_{a_{k-1}}\}.$$
Importantly $\{(\phi_{a_1}\circ \cdots \circ \phi_{a_k})(K)\}_{(a_1,\ldots,a_k)\in I_r}$ forms a cover of $K$. In \cite{Falconer ref} it is shown that when the open set condition holds, for any $x\in K$ and $r>0$ we have
\begin{equation}
\label{OSC bound}\#\{(a_1,\ldots,a_k)\in I_r:B(x,r)\cap (\phi_{a_1}\circ \cdots \circ \phi_{a_k})(K)\neq \emptyset\}\leq C 
\end{equation} 
for some $C>0$ independent of $r$. 

Fixing $x\in K$ and $r>0$ and using \eqref{self-similar scaling} and \eqref{OSC bound}, for $\delta < \frac{r}{2}$ we obtain
\begin{align} \label{self-similar LSP upper bound} 
\cH^n(B(x,r)\cap \Delta(K,\delta))&\leq \sum_{\substack{(a_1,\ldots,a_k)\in I_r\\B(x,r)\cap (\phi_{a_1}\circ \cdots \circ \phi_{a_k})(K)\neq \emptyset}}\cH^n(\Delta((\phi_{a_1}\circ \cdots \circ \phi_{a_k})(K),\delta)) \nonumber \\
&= \sum_{\substack{(a_1,\ldots,a_k)\in I_r\\B(x,r)\cap (\phi_{a_1}\circ \cdots \circ \phi_{a_k})(K)\neq \emptyset}}\cH^n((\phi_{a_1}\circ \cdots \circ \phi_{a_k})(\Delta(K,\delta\cdot (r_{a_1}\cdots r_{a_k})^{-1}))) \nonumber \\
&=\sum_{\substack{(a_1,\ldots,a_k)\in I_r\\B(x,r)\cap (\phi_{a_1}\circ \cdots \circ \phi_{a_k})(K)\neq \emptyset}} (r_{a_1}\cdots r_{a_k})^n \cH^n(\Delta(K,\delta\cdot (r_{a_1}\cdots r_{a_k})^{-1})) \nonumber \\
&\stackrel{\eqref{self-similar scaling}}\leq b_2 \sum_{\substack{(a_1,\ldots,a_k)\in I_r\\B(x,r)\cap (\phi_{a_1}\circ \cdots \circ \phi_{a_k})(K)\neq \emptyset}}(r_{a_1}\cdots r_{a_k})^n \left(\frac{\delta}{r_{a_1}\cdots r_{a_k}}\right)^{n-d} \nonumber \\
&\leq b_2 \sum_{\substack{(a_1,\ldots,a_k)\in I_r\\B(x,r)\cap (\phi_{a_1}\circ \cdots \circ \phi_{a_k})(K)\neq \emptyset}} r^d \delta^{n-d} \nonumber \\
&\stackrel{\eqref{OSC bound}}\leq b_2C r^d \delta^{n-d}.
\end{align} 
Combining \eqref{self-similar LSP lower bound} and \eqref{self-similar LSP upper bound}, we see that $K$ satisfies \eqref{Euclidean local scaling} for $\delta<r/2$. Equation \eqref{Euclidean local scaling} trivially holds for $\delta\in[r/2,r)$ since, in that case, $B(x,r)\cap \Delta(K,\delta)$ contains a ball of radius $r/2$ and is contained in a ball of radius $r$. Therefore \eqref{Euclidean local scaling} holds for all $\delta<r$ and $K$ satisfies the LSP with respect to $\ka = \frac{d}{n}$ as claimed.

\section{Preliminaries} \label{preliminaries section}
In this section we state some definitions and recall some well known facts from geometric measure theory. Throughout this paper we will say that \mbox{$f : \R^+ \to \R^+$} is a {\em dimension function} if $f$ is a left continuous, non-decreasing function such that $f(r)\to 0$ as $r\to 0 \, $. Given a ball $B:=B(x,r)$ in $X$ and a dimension function $f$, we define
\[ V^f(B)\,:=\,f(r).\] 

The Hausdorff $f$-measure with respect to the dimension function $f$ is defined as follows. Suppose $F \subset X$, let $f$ be a dimension function and let $\rho>0$. A \emph{$\rho$-cover for $F$} is any countable collection of balls $\{B_i\}_{i \in \N}$ with $r(B_i) < \rho$ for every $i \in \N$ and $F \subset \bigcup_{i \in \N}{B_i}$. We define
\[ {\cal H}^{f}_{\rho} (F)  :=  \inf \left\{ \sum_{i} V^f(B_i): \left\{B_{i} \right\} \text{is a $\rho$--cover for $F$}\right\}.\]
The {\it Hausdorff $f$-measure} of $F$ with respect to
the dimension function $f$ is then defined as
\[ {\cal H}^{f} (F) := \lim_{ \rho \rightarrow 0} {\cal H}^{f}_{\rho} (F). \]
A simple consequence of the definition of $ {\cal H}^f $ is the
following useful fact (see, for example, \cite{Falconer ref}).

\begin{lemma} \label{dimfunlemma}
Let $(X,d)$ be as above. Suppose $f$ and $g$ are dimension functions such that the ratio $f(r)/g(r) \to 0 $ as $ r \to 0 $, then ${\cal H}^{f} (F) =0 $ whenever \mbox{$\cH^g (F) < \infty $.}
\end{lemma}

When  $f(r) = r^s$ ($s \geq 0$), the measure $ \hf $
is the familiar $s$-dimensional Hausdorff measure, which we denote by $\hs$. The \emph{Hausdorff dimension}, $\dimh F$, of a set $F$ is defined as
\[ \dimh(F) \, := \, \inf \left\{ s \geq 0 : {\cal H}^{s} (F) =0 \right\}. \]

When calculating the Hausdorff dimension of a set, a usual strategy is to obtain upper and lower bounds separately. It is often the case that calculating an upper bound is relatively straightforward while determining a lower bound is much more difficult. Nevertheless, a standard tool which can frequently be employed in obtaining lower bounds for Hausdorff dimension is the following Mass Distribution Principle.

%
%
%
\begin{lemma}[Mass Distribution Principle] \label{mass distribution principle}
 Let $ \mu $ be a probability measure supported on a subset $F$ of $X$.
Suppose there are  positive constants $c$ and $r_0$ such that 
$$
\mu ( B ) \leq \, c \;  V^f(B) \;
$$ 
for any ball $B$ with radius $r \leq r_0 \, $. If $E$ is a subset of $F$ with $\mu(E) = \lambda > 0$  then $ {\cal H}^{f} (E) \geq \lambda/c$.
\end{lemma}

For this precise statement of the Mass Distribution Principle, see \mbox{\cite[Section 2]{BV MTP}.} For further general information regarding Hausdorff measures and dimension we refer the reader to \cite{Falconer ref,Mattila ref}.

Let $B:=B(x,r)$ be a ball in $(X,d)$. For any $\alpha>0$, we denote by  $\alpha B$ the ball $B$ scaled by a factor $\alpha$; i.e.  $\alpha B(x,r):= B(x, \alpha r)$. A useful covering lemma which we will use throughout is the following (see \cite{Mattila ref}).

\begin{lemma}[The $5r$-covering lemma]\label{5r}
Let $(X,d)$ be a metric space. Every family ${\cal F}$ of balls of uniformly bounded diameter in
$X$ contains a disjoint subfamily ${\cal G}$ such that
\[ \bigcup_{B \in {\cal F} } B \ \subset \ \bigcup_{B \in {\cal G} } 5B. \]
\end{lemma}

We will also make use of the following adaptation of \cite[Lemma 4]{BV Zero-one}.
 
\begin{lemma} \label{finite case lemma}
Let $B$ be a ball in the locally compact metric space $(X,d)$ and let $g$ be a doubling dimension function. Let $(S_i)_{i \in \N}$ be a sequence of subsets in $B$ and let $(\delta_i)_{i \in \N}$ be a sequence of positive numbers such that $\delta_i \to 0$ as $i \to \infty$. Let 
\[\Delta(S_i,\delta_i):=\{x \in X: d(S_i,x)<\delta_i).\]
Then, for any real number $C > 1$, 
\[\cH^g(\limsup_{i \to \infty}{\Delta(S_i,\delta_i))} = \cH^g(\limsup_{i \to \infty}{\Delta(S_i, C\delta_i))}.\]
\end{lemma}

Note that \cite[Lemma 4]{BV Zero-one} is stated in the setting of Euclidean space. Going through the steps in the proof of this lemma one can verify the above analogue holds in our setting. To prove this analogue we require a notion of the Lebesgue Density Theorem that holds for our metric space $(X,d)$ equipped with the measure $\cH^g$. Such an analogue is known to exist when $g$ satisfies our doubling bound \eqref{doubling}, see for example \cite{Rigot}.

In what follows we use the Vinogradov notation, writing $A \ll B$ if $A \leq cB$ for some positive constant $c$ and $A \gg B$ if $A \geq c'B$ for some positive constant~$c'$. If $A \ll B$ and $A \gg B$ we write $A \asymp B$ and say that $A$ and $B$ are \emph{comparable}.

\section{The $K_{G,B}$-Lemma}
\label{KGB Section}
Before proving Theorem~\ref{general mtp theorem}, we formulate suitable analogues of \cite[Lemma 4]{AB ref} and \mbox{\cite[Lemma 5]{AB ref}} which will be required in the present setting. Let 
\[\tU_j:=g^{-1}\left(\left(\frac{f(\U_j)}{g(\U_j)^{\ka}}\right)^{\frac{1}{1-\ka}}\right).\] Given a ball $B$ in $X$ and $j \in \N,$ we define
\[\Phi_j(B):=\{B(x,\tU_j)\subset B:x\in F_j\}.\]
The following is the analogue of \cite[Lemma 4]{AB ref} or \cite[Lemma 5]{BV MTP}, the so-called $K_{G,B}$-Lemma, we obtain in the setting currently under consideration.

\begin{lemma} \label{kgb lemma}
Let $(X,d)$, $\cF$, $\U$, $g$ and $f$ be as given in Theorem~\ref{general mtp theorem} and assume that the hypotheses of Theorem~\ref{general mtp theorem} hold. Then, for any ball $B$ in $X$ and any $G \in \N$, there exists a finite collection 
\[K_{G,B} \subset \{(A;j): j \geq G, A \in \Phi_j(B)\}\]
satisfying the following properties:
\begin{enumerate}[{\rm(i)}]
\item{if $(A;j) \in K_{G,B}$ then $3A \subset B$;}
\item{if $(A;j), (A',j') \in K_{G,B}$ are distinct then $3A \cap 3A' \neq \emptyset$; and}
\item{there exists a constant $c_5\geq 0$ independent of our choice of ball $B$ such that$$\displaystyle{\cH^g\left(\bigcup_{(A;j) \in K_{G,B}}{A}\right) \geq c_5 \cH^g(B).}$$ }
\end{enumerate}
\end{lemma}

Similarly to \cite[Lemma 4]{AB ref}, the collection $\kgb$ here is a collection of balls drawn from the families $\Pj(B)$. These balls correspond to the $\limsup$ set $\La(\tU)$. From each of these balls what we are actually interested in is extracting a suitable collection of balls corresponding to the $\limsup$ set $\La(\U)$. We adopt the notation from~\cite{AB ref} and write $\aj$ for a generic ball from $\kgb$ to ``remember'' the index $j$ of the family $\Pj(B)$ that the ball $A$ comes from, but just write $A$ if we are referring only to the ball $A$ (as opposed to the pair $\aj$). Making such a distinction is necessary for us to be able to choose the ``right'' collection of balls within $A$ that at the same time lie in an $\U_j$-neighbourhood of the relevant $F_j$. Indeed, for $j\neq j'$ we could have $A = A'$ for some $A \in \Pj(B)$ and $A' \in \Phi_{j'}(B)$.


\begin{proof}[Proof of Lemma \ref{kgb lemma}]
For $j \in \N$ and a fixed ball $B$ in $X$, consider the set of balls
\[\Pj^3(B) := \{B(x,3\tU_j) \subset B: x \in F_j\}.\]
It follows from our assumption (\ref{full measure assumption}) that for any $G\ge1$ we have
\[ \cH^g\left(\bigcup_{j \geq G}{(\Delta(F_j, 3\tU_j) \cap B)}\right) = \cH^g(B).\]
Observe that $\tU_j \to 0$ as $j \to \infty$ because $g$ and $f/g^{\ka}$ are dimension functions. Therefore for $j \in \N$ sufficiently large, 
\[{\bigcup_{L \in \Pj^3(B)}{L}} \supset \Delta(F_j, 3\tU_j) \cap \frac{1}{2}B.\]
Therefore for any sufficiently large $G \in \N$, we have 
\[\cH^g\left(\bigcup_{j \geq G}{\bigcup_{L\in \Pj^3(B)}{L}}\right) \geq \cH^g\left(\bigcup_{j \geq G}{\left(\Delta(F_j, 3\tU_j) \cap \frac{1}{2}B\right)}\right) = \cH^g\left(\frac{1}{2}B\right).
\]
Suppose $G' \in \N$ is large enough that the above inequality holds for any $G \geq G'$. Clearly for any $G < G'$ we also have 
\[\bigcup_{j \geq G}{\bigcup_{L \in \Pj^3(B)}{L}} \supset \bigcup_{j \geq G'}{\bigcup_{L \in \Pj^3(B)}{L}}.\]
Therefore for any $G \in \N$ it follows that
\begin{align} \label{kgb lemma Hausdorff}
\cH^g\left(\bigcup_{j \geq G}{\bigcup_{L \in \Pj^3(B)}{L}}\right) &\geq \cH^g\left(\frac{1}{2}B\right).
\end{align}
%
Next, by the $5r$-covering Lemma (Lemma \ref{5r}), there exists a disjoint subcollection \mbox{$\cG \subset \{\Lj:j\ge G,~L\in \Pj^3(B)\}$} satisfying
\[ \bigcup_{(L;j) \in \cG}^{\circ}{L} \subset \bigcup_{j \geq G}{\bigcup_{L \in \Pj^3(B)}{L}} \subset \bigcup_{(L;j) \in \cG}{5L}.\]

Let $\cG':=\{(\tfrac13L;j):(L;j)\in\cG\}$ be the balls from the collection $\cG$ all scaled by a factor of $1/3$. Note that the balls in $\cG'$ are still disjoint when scaled by 3. By the above, we have that
\begin{align} \label{kgb lemma inclusions}
\bigcup_{\aj \in \cG'}^{\circ}{A} ~\subset~ \bigcup_{j \geq G}~{\bigcup_{L\in \Pj^3(B)}{L}} \subset \bigcup_{\aj \in \cG'}{15A}.
\end{align}
It follows from \eqref{doubling}, \eqref{gball}, and the disjointness of the balls in $\cG'$ that
\begin{align*}
\cH^g\left(\bigcup_{\aj \in \cG'}{A}\right) = \sum_{\aj \in \cG'}{\cH^g(A)} &\stackrel{\eqref{gball}}\asymp  \sum_{\aj \in \cG'}g(r(A))\\
&\stackrel{\eqref{doubling}}\gg  \sum_{\aj \in \cG'}g(r(15A))\\
&\stackrel{\eqref{gball}}\asymp \sum_{\aj \in \cG'}\cH^g(15A) \\
&\geq \cH^g\left(\bigcup_{\aj \in \cG'}{15A}\right).
\end{align*}
Now also using (\ref{kgb lemma Hausdorff}) and (\ref{kgb lemma inclusions}), we see that
\begin{align*}
\cH^g\left(\bigcup_{\aj \in \cG'}{A}\right) &\stackrel{\eqref{kgb lemma inclusions}}\gg \cH^g\left(\bigcup_{j \geq G}{\bigcup_{L \in \Pj^3(B)}{L}}\right) \\
                                              &\stackrel{\eqref{kgb lemma Hausdorff}} \geq \cH^g\left(\frac{1}{2}B\right) \\
                                              &\stackrel{\eqref{gball}}\asymp g(r(\frac{1}{2}B))\\
                                              &\stackrel{\eqref{doubling}}\asymp g(r(B))\\
                                              &\stackrel{\eqref{gball}}\asymp \cH^g(B).
\end{align*}
Thus, there exists a constant $c'>0$ such that 
$$c'\cdot \cH^g(B)\leq  \cH^g\left(\bigcup_{\aj \in \cG'}{A}\right).$$

Since the balls in $\cG'$ are disjoint and contained in $B$, it follows that
\[ \cH^g\left({\bigcup_{\substack{\aj \in \cG' \\ j \geq N}}A}\right) \to 0 \quad \text{as } N \to \infty.\]
Consequently, there must exist $N_0 \in \N$ such that
\[ \cH^g\left({\bigcup_{\substack{\aj \in \cG' \\ j \geq N_0}}A}\right)  < \frac{c'}{2} \cH^g(B).\]
We define $\kgb$ to be the subcollection of $\aj \in \cG'$ with $G\le j < N_0.$ By the above we see that $\kgb$ is a finite collection of balls while still satisfying the required properties (i)--(iii) with $c_5 = \frac{c'}{2}$.
\end{proof}

As mentioned previously, from each of the balls in $K_{G,B}$ we wish to extract a collection of balls corresponding to $\La(\U)$. The desired properties and existence of such collections are summarised in the following lemma, which constitutes the required analogue of \cite[Lemma 5]{AB ref} in this setting.

\begin{lemma} \label{C(A;n) lemma}
Let $(X,d)$, $\cF$, $\U$, $f$, $g$, and $B$ be as in Lemma~\ref{kgb lemma} and assume that the hypotheses of Theorem~\ref{general mtp theorem} hold. Furthermore, assume that $f(r)/g(r) \to \infty$ as $r \to 0$. Let $\kgb$ be as in Lemma~\ref{kgb lemma}. Then, provided that $G$ is sufficiently large, for any $\aj\in\kgb$ there exists a collection $\caj$ of balls satisfying the following properties:
\begin{enumerate}[{\rm(i)}]
\item{each ball in $\caj$ is of radius $\U_j$ and is centred on $F_j;$}
\item{if $L \in \caj$ then $3L \subset A;$}
\item{if $L, M \in \caj$ are distinct then $3L \cap 3M = \emptyset;$}
\item{~$\displaystyle \cH^g\big(\Delta(F_j,\U_j)\cap \tfrac12A \big)~\ll~\cH^g\left(\bigcup_{L \in \caj}{L}\right) ~\leq~ \,\cH^g\big(\Delta(F_j,\U_j)\cap A \big)$; and}
\item there exist some constants $d_1,d_2 > 0$ such that 
\begin{align} \label{cardinality of C(A;n)}
d_1\times\left(\frac{f(\U_j)}{g(\U_j)}\right)^{\frac{\ka}{1-\ka}} ~\leq~ \# \caj &~\leq~ d_2\left(\frac{f(\U_j)}{g(\U_j)}\right)^{\frac{\ka}{1-\ka}}.
\end{align}
\end{enumerate}
\end{lemma}

\begin{proof}
We begin by showing that
\begin{equation}\label{conv0}
  \frac{\U_j}{\tU_j}\to 0\qquad\text{as }j\to\infty.
\end{equation}
To this end, suppose that $N \in \N$. We aim to show that for all sufficiently large $j \in \N$ we have
\begin{equation}\label{conv1}
\U_j < \frac{\tU_j}{2^N}.
\end{equation}
Observe that \eqref{conv1} holds if
\begin{align}\label{conv2}
g(\U_j) < g\left(\frac{\tU_j}{2^N}\right).
\end{align}
Furthermore, by repeated application of \eqref{doubling}, we see that 
\[g\left(\frac{\tU_j}{2^N}\right) > \frac{g(\tU_j)}{\lambda^N} = \frac{1}{\lambda^N}\left(\frac{f(\U_j)}{g(\U_j)^{\ka}}\right)^{\frac{1}{1-\ka}},\]
where $\lambda$ is the doubling constant. Consequently, \eqref{conv2} holds if 
\[g(\U_j) < \frac{1}{\lambda^N}\left(\frac{f(\U_j)}{g(\U_j)^{\ka}}\right)^{\frac{1}{1-\ka}}.\]
Rearranging the above we get
\[\lambda^{N(1-\ka)} < \frac{f(\U_j)}{g(\U_j)}.\] 
By the assumptions that $\frac{f(r)}{g(r)} \to \infty$ as $r \to 0$ and $\U_j \to 0$ as $j \to \infty$, we see that this inequality holds for sufficiently large $j \in \N$, thus verifying \eqref{conv0}. 

In light of \eqref{conv0} we can assume that $G$ is sufficiently large so that
\begin{equation}\label{sep0}
\text{$6\U_j<\tU_j$ \qquad for any $j\ge G$.}
\end{equation}
Let $ x_1,\dots, x_t\in F_j\cap\tfrac12A$ be a maximal collection of points such that
\begin{equation}\label{sep}
d(x_i,x_{i'})> 6\U_j\qquad\text{if }i\neq i'.
\end{equation}
Define $\caj$ to be the collection of balls
\[ \cC\aj:=\{B( x_1,\U_j),\dots,B( x_t,\U_j)\}\,.\]

By construction, property (i) is satisfied by the collection $\caj$. Next, recall that $A\in\Phi_j(B)$ and so $\tfrac12A$ has radius $\tfrac12\tU_j$.
If $L:=B( x_i,\U_j)\in \cC\aj$ then any $ y\in 3L$ satisfies $d(y,x_i)<3\U_j$. Supposing $x_0$ is the centre of $A$, we also have $d(x_i,x_0)\le\tfrac12\tU_j$. Combining \eqref{sep0} and the triangle inequality we obtain
\[d(y,x_0)\le d(y,x_i)+d(x_i,x_0)\le 3\U_j+\tfrac12\tU_j<\tU_j.\] 
Therefore property~(ii) follows. In addition property~(iii) follows from \eqref{sep}.

It is a consequence of the maximality of $ x_1,\dots,  x_t$ that for any $ x \in F_j\cap\tfrac12A,$ there exists an $ x_i$ from this collection such that $d(x,x_i)\le 6\U_j$. Consequently
\[\Delta(F_j,\U_j)\cap \tfrac12A~\subset~\bigcup_{L\in\cC\aj}7L.\]
Therefore, by \eqref{doubling} and \eqref{gball},
\begin{align*}
\cH^g(\Delta(F_j,\U_j)\cap \tfrac12A)&\leq \cH^g\left(\bigcup_{L \in \caj}{7L}\right) \\
                                     &\le\sum_{L\in\cC\aj}\cH^g(7L)\\
                                     &\ll\sum_{L\in\cC\aj}\cH^g(L)\\
                                     &\ll\cH^g\left(\bigcup^\circ_{L\in\cC\aj}L\right).
\end{align*}
However, by property~(ii), we have
\[\bigcup^\circ_{L\in\cC\aj}L~\subset~\Delta(F_j,\U_j)\cap A.\]
This together with the previous inequality proves property~(iv).

As a byproduct of (\ref{sep0}) we have that $\U_j < \frac{\tU_j}{2}$ for all $j \geq  G.$ Therefore, by the LSP we have 
\[\cH^g\left(\Delta(F_j,\U_j)\cap \frac{1}{2}A \right) ~\asymp~ g\left(\frac{\tU_j}{2}\right)^{\ka}\cdot g(\U_j)^{1-\ka}.\]
Combining this with the doubling property \eqref{doubling} and the LSP we obtain
\begin{align} \label{scaling consequence}
\cH^g\left(\Delta(F_j,\U_j)\cap \frac{1}{2}A \right) ~\asymp~ g(\tU_j)^{\ka}\cdot g(\U_j)^{1-\ka} ~\asymp~ \cH^g\left(\Delta(F_j,\U_j)\cap A \right).
\end{align}
By \eqref{gball} and the disjointness of the balls in $\caj$ we have
\[\cH^g\left(\bigcup_{L \in \caj}{L}\right) ~=~ \sum_{L \in \caj}\cH^g\left(L\right) ~\asymp~ \sum_{L \in \caj}g(\U_j) = \#\caj\,g(\U_j).\] 
Combining the above with (\ref{scaling consequence}) and property (iv) we get
\[\# \caj \asymp \frac{ g(\tU_j)^{\ka}\cdot g(\U_j)^{1-\ka}}{g(\U_j)} = \left(\frac{ g(\tU_j)}{g(\U_j)}\right)^{\ka} = \left(\frac{f(\U_j)}{g(\U_j)}\right)^{\frac{\ka}{1-\ka}}\]
So property (v) holds.
\end{proof}

\section{Proof of Theorem~\ref{general mtp theorem}}
\label{proof section}
\subsection{Strategy}
Fix an arbitrary ball $B_0$ in $X$ and suppose the assumptions of Theorem~\ref{general mtp theorem} hold. Our goal is to show that
\begin{align} \label{ultimate aim}
\cH^f(B_0 \cap \Lambda(\U))=\cH^f(B_0).
\end{align}
Since $f/g$ is monotonic, there are three situations to consider:
\begin{enumerate}[(a)]
\item{$\frac{f(r)}{g(r)} \to \infty$ as $r \to 0$;} \label{infinite}
\item{$\frac{f(r)}{g(r)} \to 0$ as $r \to 0$; and} \label{zero}
\item{$\frac{f(r)}{g(r)} \to \ell$ as $r \to 0$, where $0 < \ell < \infty$.} \label{constant}
\end{enumerate}

If we are in situation \eqref{zero} it follows from Lemma~\ref{dimfunlemma} that $\cH^f(B_0) = 0$. Since $B_0 \cap \Lambda(\U) \subset B_0$ the result follows.

If we are in case \eqref{constant} it can be shown that $\tU_j\asymp \U_j$. It then follows from Lemma~\ref{finite case lemma} that $\cH^g(B_0\cap \Lambda(\tU))=\cH^g(B_0\cap \Lambda(\U))$. In turn, it follows from~\eqref{full measure assumption} that $\cH^g(B_0\cap \Lambda(\U))=\cH^g(B_0)$. Finally, the proof is completed in this case by noting that $\cH^f=\ell \cdot \cH^g$ and therefore $$\cH^f(B_0\cap \Lambda(\U))=\ell \cdot \cH^g(B_0\cap \Lambda(\U))= \ell \cdot  \cH^g(B_0)= \cH^f(B_0).$$

It remains to address case \eqref{infinite}. Thus, from now on we will assume that $\frac{f(r)}{g(r)} \to \infty$ as $r \to 0$. In this case, it is a consequence of Lemma~\ref{dimfunlemma} that \mbox{$\cH^f(B_0) = \infty$.} So, to prove Theorem~\ref{general mtp theorem} it suffices to show that 
$$\cH^f(B_0 \cap \Lambda(\U))=\infty.$$ 
To achieve this goal we will show that for any $\eta>1$, we can construct a Cantor set $\mathbb{K}_\eta$ contained in $B_0 \cap \Lambda(\U)$ which supports a probability measure $\mu$ satisfying 
\begin{equation}
\label{Wanttoshow}
\mu(D)\ll \frac{V^f(D)}{\eta},
\end{equation} 
 for all balls $D$ with sufficiently small radii, where the implicit constants are independent of $D$ and $\eta$. The result then follows from the Mass Distribution Principle (Lemma~\ref{mass distribution principle}) upon taking $\eta$ to be arbitrarily large since the Mass Distribution Principle yields $\cH^f(\K_{\eta}) \geq \eta$ and $\K_{\eta} \subset B_0 \cap \La(\U)$.  

\subsection{Desired properties of $\K_\eta$} \label{properties section}
The construction of the Cantor set we present here is an adaptation of that given in \cite{AB ref} and \cite{BV MTP}. For ease of comparison we will generally adopt the notation used in \cite{AB ref}. 

Fix $\eta > 1$. Our Cantor set $\K_\eta$ will take the form
$$\K_\eta=\bigcap_{n=1}^{\infty} \K(n)$$
where $\K(n)\supset \K(n+1)$. The fact that $(X,d)$ is a locally compact metric space guarantees that $\K_\eta$ is non empty. 

Each level $\K(n)$ of the Cantor set will be a union of balls and we will denote the corresponding set of level $n$ balls by $K(n)$. For each ball $B\in K(n-1)$ we will construct an \emph{$(n,B)$-local level}, henceforth denoted by $K(n,B)$, which will consist of balls contained in $B$. The set of level $n$ balls, $K(n),$ will then be defined by
$$K(n):=\bigcup_{B\in K(n-1)}K(n,B).$$ 
Each $(n,B)$-local level will be constructed of \emph{local sub-levels} and will take the form 
\begin{align} \label{local level}
K(n,B):=\bigcup_{i=1}^{l_B}K(n,B,i),
\end{align}
where $K(n,B,i)$ denotes the $i$th local sub-level and $l_B$ is the number of local sub-levels forming $K(n,B)$. 
What is more, each local sub-level will take the form
\begin{equation}
\label{local sub-level}
K(n,B,i):=\bigcup_{B'\in \cG(n,B,i)}\bigcup_{(A;j)\in K_{G',B'}}\caj.
\end{equation}
Here, $\cG(n,B,i)$ will be a suitable collection of balls contained in $B$ and, for each ball $B' \in \cG(n,B,i)$, $K_{G',B'}$ will be the corresponding finite collection whose existence is asserted by Lemma~\ref{kgb lemma}. The collections $\caj$ will be those arising from Lemma~\ref{C(A;n) lemma}. The set of pairs $(A;j)$ included in \eqref{local sub-level} will be denoted by $\tilde K(n,B,i)$. As such
\begin{align} \label{local level components}
\tilde K(n,B,i):=\bigcup_{B'\in \cG(n,B,i)}K_{G',B'}\quad\textrm{and}\quad K(n,B,i)=\bigcup_{(A;j)\in \tilde K(n,B,i)}\caj.
\end{align}
We will also require that $\K_\eta$ satisfies the following properties.

\subsection*{The properties of levels and sub-levels of $\K_\eta$}

\begin{enumerate}

\item[{\bf(P0)}] $K(1)=\{B_0\}$.

\medskip
	
\item[{\bf(P1)}] For any $n\geq 2$ and $B\in K(n-1)$ the balls $$\{3L: L\in K(n,B)\}$$ are disjoint and contained in $B$.

\medskip

\item[{\bf(P2)}] For any $n\geq 2$, $B\in K(n-1),$ and $i\in\{1,\ldots, l_B\},$ the local sub-level $K(n,B,i)$ is a finite union of some collections $C(A;j)$ of balls satisfying properties (i)--(v) of Lemma~\ref{C(A;n) lemma}. Moreover, the balls $3A$ are disjoint and contained in $B$.

\medskip

\item[{\bf(P3)}] For any $n\geq 2$, $B\in K(n-1),$ and $i\in\{1,\ldots,l_B\},$ we have 
$$\sum_{(A;j)\in \tilde K(n,B,i)}V^g(A) \geq c_6 V^g(B)$$
where 
\[c_6:=\frac{1}{2\lambda} \left(\frac{c_1}{c_2}\right)^2\frac{c_5}{c_7}.\] 
The constants $c_1$ and $c_2$ are those appearing in \eqref{gball}, $c_5$ comes from Lemma~\ref{kgb lemma}~(iii), $\lambda$ is the doubling constant associated with $g$, and $c_7$ is a fixed constant such that 
\[g(r(5B)) \leq c_7g(r(B)) \]
for any ball $B$ in $X$. Note that the existence of $c_7$ is guaranteed by the doubling property \eqref{doubling}.
	
\medskip
	
\item[{\bf(P4)}] For any $n\geq 2$, $B\in K(n-1)$, $i\in \{1,\ldots, l_B -1\}$, $L\in K(n,B,i),$ and $M\in K(n,B,i+1),$ we have 
\[f(r(M))\leq \frac{f(r(L))}{2} \quad \textrm{ and } \quad \frac{f(r(M))}{g(r(M))^{\ka}}\leq \frac{f(r(L))}{2g(r(L))^{\ka}}.\]

\medskip
	
\item[{\bf(P5)}] The number of sub-levels is defined by 
\[ l_B := \left\{ \begin{array}{ll}
	\left[\dfrac{c_2 \eta}{c_6 \cH^g(B)}\right]+1 & \mbox{if $B=B_0:=\K(1)$};\\[5ex]
		\left[\dfrac{V^f(B)}{c_6 V^g(B)}\right]+1 & \mbox{if $B\in K(n)$ with $n\geq 2$}.\end{array} \right. \] 
and $l_B\geq 2$ for $B\in K(n)$ with $n\geq 2$.
\end{enumerate}

\subsection{Existence of $\K_\eta$}
We now prove that it is possible to construct a set $\K_\eta \subset B_0 \cap \La(\U)$ satisfying properties {\bf(P0)--(P5)}. To this end, let
\begin{align} \label{K_l notation}
K_l(n,B):=\bigcup_{i=1}^l K(n,B,i) \quad \textrm{ and } \quad \tilde K_l(n,B):= \bigcup_{i=1}^l \tilde K(n,B,i).
\end{align}

\noindent \textbf{Level 1.} Let $\K(1):=B_0$ so {\bf(P0)} holds. 

All other levels of $\K_\eta$ are defined inductively. Therefore, assume levels $\K(1),\ldots, \K(n-1)$ have been constructed. To construct the $n$th level, we need to construct $(n,B)$-local levels for all balls $B \in K(n-1)$.

\noindent \textbf{Level n.} Fix $B\in K(n-1)$ and let $\varepsilon:=\varepsilon(B)$ be a small constant which will be explicitly determined later. Let $G$ be sufficiently large so Lemma~\ref{kgb lemma} and Lemma~\ref{C(A;n) lemma} can be invoked. We may also assume that $G$ is large enough that

\begin{equation}\label{G1}
3g(\U_j)^{1-{\ka}}<\frac{f(\U_j)}{g(\U_j)^{\ka}}\, \quad \text{for all } j\geq G,
\end{equation}
\begin{equation}\label{G2}
\frac{g(\U_j)}{f(\U_j)}<\varepsilon\, \frac{g(r(B))}{f(r(B))}\, \quad \text{for all } j\geq G,
\end{equation}
and
\begin{equation}
\label{G3}
\left[\frac{f(\U_j)}{c_6 g(\U_j)}\right]\geq 1\, \quad \text{for all } j\geq G.
\end{equation}
Here, $c_6$ is the constant appearing in {\bf(P3)}. Inequalities \eqref{G1}--\eqref{G3} are achievable since $f(r)/g(r)\to \infty$ as $r\to 0.$ 

Recall that the $(n,B)$-local level, $K(n,B)$, consists of local sub-levels. These are defined as follows.

\noindent\textbf{Sub-level 1.} For $B$ and $G$ as above let $K_{G,B}$ be the collection of balls arising from Lemma~\ref{kgb lemma}. We define the first sub-level of $K(n,B)$ to be 
$$K(n,B,1):=\bigcup_{(A;j)\in K_{G,B}} C(A;j).$$
Hence, 
$$\tilde K(n,B,1)= K_{G,B} \quad \textrm{ and } \quad \cG(n,B,1)=\{B\}.$$

\noindent \textbf{Higher sub-levels.}
The higher sub-levels are defined inductively. Suppose the first $l$ sub-levels $K(n,B,1),\ldots, K(n,B,l)$ have been constructed and properties {\bf(P1)--(P4)} hold with $l$ in place of $l_B$. Since we require fairly stringent separation conditions between balls in $\K_{\eta}$, we first verify that there is ``space'' left over in $B$ for the sub-level $K(n,B,l+1)$ after the first $l$ sub-levels, $K(n,B,1),\dots,K(n,B,l)$, have been constructed. Let $$A^{(l)}:=\frac{1}{2}B \setminus \bigcup_{L\in K_{l}(n,B)}4L.$$ We will show that 
\begin{align} \label{enoughspace}
\cH^g(A^{(l)})\geq \frac{1}{2}\cH^g\left(\frac{1}{2}B\right).
\end{align}
Using \eqref{doubling}, \eqref{gball}, and the upper bound for $\#\caj$ given in \eqref{cardinality of C(A;n)}, we obtain
\begin{align*}
\cH^g\left(\bigcup_{L\in K_{l}(n,B)}4L\right)&\leq \sum_{L\in K_{l}(n,B)} \cH^g(4L)\\
                                             &\stackrel{\eqref{gball}}\leq \sum_{L\in K_{l}(n,B)} c_2g(r(4L))\\
                                             &\stackrel{\eqref{doubling}}\leq \sum_{L\in K_{l}(n,B)} c_2 \lambda^2 g(r(L))\\
                                             &= \sum_{i=1}^l \sum_{L\in K(n,B,i)} c_2 \lambda^2 g(r(L))\\
                                             &= \sum_{i=1}^l \sum_{(A;j)\in \tilde K(n,B,i)} c_2 \lambda^2 \# C(A;j) g(\U_j)\\
                                             &\stackrel{\eqref{cardinality of C(A;n)}}\leq \sum_{i=1}^l \sum_{(A;j)\in \tilde K(n,B,i)} c_2 \lambda^2 d_2 \left(\frac{f(\U_j)}{g(\U_j)}\right)^{\frac{\ka}{1-\ka}} g(\U_j)\\
                                             &= c_2 \lambda^2 d_2 \sum_{i=1}^l \sum_{(A;j)\in \tilde K(n,B,i)} \frac{f(\U_j)^{\frac{1}{1-\ka}}}{g(\U_j)^{\frac{\ka}{1-\ka}}}\cdot \frac{g(\U_j)}{f(\U_j)}.
\end{align*}
Recalling condition \eqref{G2} and property {\bf(P2)} we further see that 
\begin{align} \label{space}
\cH^g\left(\bigcup_{L\in K_{l}(n,B)}4L\right)&\stackrel{\eqref{G2}} < \varepsilon c_2 \lambda^2 d_2 \cdot \frac{g(r(B))}{f(r(B))} \sum_{i=1}^l \sum_{(A;j)\in \tilde K(n,B,i)} \frac{f(\U_j)^{\frac{1}{1-\ka}}}{g(\U_j)^{\frac{\ka}{1-\ka}}} \nonumber \\
                                             &= \varepsilon c_2 \lambda^2 d_2 \cdot \frac{g(r(B))}{f(r(B))} \sum_{i=1}^l \sum_{(A;j)\in \tilde K(n,B,i)} g(\tU_j) \nonumber \\
                                             &\stackrel{\eqref{gball}}\leq  \varepsilon \frac{c_2 \lambda^2 d_2}{c_1} \cdot \frac{g(r(B))}{f(r(B))} \sum_{i=1}^l \sum_{(A;j)\in \tilde K(n,B,i)} \cH^g(A) \nonumber \\
                                             &\stackrel{{\bf(P2)}}\leq  \varepsilon \frac{c_2 \lambda^2 d_2}{c_1} \cdot \frac{g(r(B))}{f(r(B))}l \cH^g(B) \nonumber \\
                                             &\leq \varepsilon \frac{c_2 \lambda^2 d_2}{c_1} \cdot \frac{g(r(B))}{f(r(B))}(l_B-1) \cH^g(B) \nonumber \\
                                             &\stackrel{\eqref{gball}} \leq \varepsilon \frac{c_2^2 \lambda^2 d_2}{c_1} \cdot \frac{g(r(B))^2}{f(r(B))}(l_B-1).
\end{align}
To establish \eqref{enoughspace} we will show that
\[\cH^g\left(\bigcup_{L\in K_{l}(n,B)}4L\right) < \frac{1}{2} \cH^g\left(\frac{1}{2}B\right).\]
By \eqref{gball} we have 
\[\frac{1}{2} \cH^g\left(\frac{1}{2}B\right) \geq \frac{c_1}{2}g\left(\frac{1}{2}B\right)\]
and by \eqref{doubling} we have 
\[g\left(\frac{1}{2}B\right) > \frac{1}{\lambda}g(r(B)).\]
Therefore it suffices to show that
\begin{align} \label{enoughtoshow}
\cH^g\left(\bigcup_{L\in K_{l}(n,B)}4L\right) < \frac{c_1}{2\lambda}g(r(B)).
\end{align}
It follows from \eqref{space} that \eqref{enoughtoshow} is implied by
\begin{equation}
\label{willshow}
\varepsilon \frac{c_2^2 \lambda^2 d_2}{c_1} \cdot \frac{g(r(B))^2}{f(r(B))}(l_B-1)<\frac{c_1}{2\lambda}g(r(B)).
\end{equation}
Taking $$\varepsilon(B):=\frac{c_1}{4\lambda}\Big(\frac{c_2^2 \lambda^2 d_2}{c_1} \cdot \frac{g(r(B))}{f(r(B))}(l_B-1)\Big)^{-1}$$ 
we see that \eqref{willshow} is satisfied and, thus, \eqref{enoughtoshow} and \eqref{enoughspace} both hold.


Next, observe that the quantity 
$$d_{min}:=\min\{r(L):L\in K_l(n,B)\}$$
is well-defined and positive since the collection $K_l(n,B)$ is finite. Let 
$$\cA(n,B,l):=\{B(x,d_{\min}):x\in A^{(l)}\}.$$ 
By Lemma~\ref{5r} there exists a disjoint subcollection $\cG(n,B,l+1)$ of $\cA(n,B,l)$ such that
\begin{align} \label{inclusions}
A^{(l)}\subset \bigcup_{B'\in \cA(n,B,l)} B'\subset \bigcup_{B'\in \cG(n,B,l+1)} 5B'.
\end{align}
Note that each element of the collection $\cG(n,B,l+1)$ is a a subset of $B$. Since the balls in this collection are disjoint and all have the same radius, $\cG(n,B,l+1)$ must be finite. Furthermore, by our construction,
\begin{align} \label{disjointness}
B'\cap\bigcap_{L\in K_{l}(n,B)}3L=\emptyset \quad \text{for any } B'\in \cG(n,B,l+1).
\end{align}
By the above \eqref{inclusions} and \eqref{enoughspace} we have
\begin{align} \label{inclusion consequence}
\cH^g\left(\bigcup_{B'\in \cG(n,B,l+1)} 5B'\right)\geq \cH^g(A^{(l)})\geq \frac{1}{2}\cH^g\left(\frac{1}{2}B\right).
\end{align}

Since ${\cal G}(n,B,l+1)$ is a disjoint collection of balls we have the following
\begin{align*}
\cH^g \left( \bigcup_{B' \in {\cal G}(n,B,l+1) } 5 B' \right) &\leq \sum_{B' \in {\cal G}(n,B,l+1)}{\cH^g(5 B')} \\[3ex]
&\stackrel{\eqref{gball}} \leq c_2 \sum_{{B'\in {\cal G}(n,B,l+1)}}{g(r(5B'))} \\[3ex]
&\leq c_2 c_7 \sum_{{B'\in {\cal G}(n,B,l+1)}}{g(r(B'))} \\
&\stackrel{\eqref{gball}} \leq \frac{c_2c_7}{c_1} \sum_{{B'\in {\cal G}(n,B,l+1)}}{\cH^g(B')} \\
&= \frac{c_2c_7}{c_1} \cH^g\left(\bigcup^\circ_{B' \in {\cal G}(n,B,l+1)}  B' \right).
\end{align*}
Combining this with \eqref{inclusion consequence} from above implies
\begin{equation}\label{bl Hausdorff}
\cH^g \left(\bigcup^\circ_{B' \in {\cal G}(n,B,l+1) } B'
\right)\,\ge\,\frac{c_1}{2 c_2 c_7} \ \
\cH^g\Big(\frac{1}{2}B\Big)\,.
\end{equation}

Now, to construct the $(l+1)$th sub-level $K(n,B,l+1)$, let $G'\geq G$ be sufficiently large so that we can apply Lemmas~\ref{kgb lemma} and \ref{C(A;n) lemma} to each ball \mbox{$B'\in\cG(n,B,l+1)$.} Moreover, we assume that $G'$ is sufficiently large so that for every $j \geq G'$,
\begin{align} \label{f and g relations}
f(\U_j) \leq \frac{1}{2}\min_{L \in K_l(n,B)}{f(r(L))} \quad \text{and} \quad \frac{f(\U_j)}{g(\U_j)^{\ka}} \leq \frac{1}{2}\min_{L \in K_l(n,B)}\frac{f(r(L))}{g(r(L))^{\ka}} .
\end{align}
Such a $G'$ exists since there are only finitely many balls in $K_l(n,B)$, $\U_j \to 0$ as $j \to \infty$, and because $f$ and $f/g^{\ka}$ are dimension functions.

To each ball $B'\in \cG(n,B,l+1)$ we apply Lemma~\ref{kgb lemma} to obtain a collection of balls $K_{G',B'}.$ We then define
\[ K(n,B,l+1) := \bigcup_{B'\in \cG(n,B,l+1)}~{\bigcup_{\aj \in K_{G',B'}}}~\caj.\]
Consequently,
$$
\widetilde K(n,B,l+1)=\bigcup_{B'\in\cG(n,B,l+1)}~~K_{G',B'}\,.
$$
As $G' \geq G$, properties (\ref{G1})--(\ref{G3}) remain valid. We now verify that properties \textbf{(P1)}--\textbf{(P5)} hold for this local sub-level.

To prove \textbf{(P1)} holds we first observe that it is satisfied for balls in $\bigcup_{\aj \in K_{G',B'}}\caj$ by the properties of $\caj$ and the fact that the balls in $K_{G',B'}$ are disjoint. The balls in $K_{G',B'}$ are by definition contained in $B'$ and the balls $B' \in \cG(n,B,l+1)$ are disjoint, therefore \textbf{(P1)} is satisfied for all balls $L$ in $K(n,B,l+1)$. Last of all, combining this observation with (\ref{disjointness}) we can conclude that \textbf{(P1)} is satisfied for all balls $L$ in $K_{l+1}(n,B)$. Property \textbf{(P2)} is satisfied for this sub-level because of Lemma~\ref{kgb lemma} (i) and (ii) and because the balls $B' \in \cG(n,B,l+1)$ are disjoint.

We now prove that \textbf{(P3)} still holds for $i = l+1$. Recalling \eqref{gball}, we have
\begin{align*}
\sum_{\aj\in \widetilde K(n,B,l+1)}{V^g(A)}
&= \sum_{B' \in \cG(n,B,l+1)}~{\sum_{\aj \in K_{G',B'}}{V^g(A)}} \\
&\stackrel{\eqref{gball}}\geq \frac{1}{c_2}\sum_{B' \in \cG(n,B,l+1)}~{\sum_{\aj \in K_{G',B'}}{\cH^g(A)}}. \\
\end{align*}
Combining this with Lemma~\ref{kgb lemma}~(iii) and the fact that the balls in $\cG(n,B,l+1)$ are disjoint, we see that
\begin{align*}
\sum_{\aj\in \widetilde K(n,B,l+1)}{V^g(A)}
&\geq \frac{1}{c_2}\sum_{B' \in \cG(n,B,l+1)}{c_5\cH^g\left(B'\right)} \\[1ex]
&= \frac{c_5}{c_2} \cH^g\left(\bigcup_{B' \in \cG(n,B,l+1)}{B'}\right) \\[1ex]
&\stackrel{\eqref{bl Hausdorff}}\geq \frac{c_1}{2c_2c_7}\frac{c_5}{c_2}\cH^g\left(\tfrac{1}{2}B\right) \\[1ex]
&\stackrel{\eqref{gball}}\geq \frac{c_1}{2c_2c_7}\frac{c_5}{c_2}c_1g\left(r(\tfrac{1}{2}B)\right) \\[1ex]
&\stackrel{\eqref{doubling}}\geq \frac{1}{2\lambda}\left(\frac{c_1}{c_2}\right)^2 \frac{c_5}{c_7}g(r(B)) \\[1ex]
&= c_6 V^g(B).
\end{align*}
Property \textbf{(P4)} is satisfied because of (\ref{f and g relations}). Finally, property \textbf{(P5)}, that $l_L \geq 2$ for any ball $L$ in $K(n,B,l+1)$, follows from (\ref{G3}).

Therefore properties \textbf{(P1)}--\textbf{(P5)} are satisfied up to the local sub-level $K(n,B,l+1).$ This establishes the existence of the local level \mbox{$K(n,B) = K_{l_B}(n,B)$} for each $B \in K(n-1)$. This then establishes the existence of the $n$th level $K(n)$ (and also $\K(n)$).

\subsection{The measure $\mu$ on $\K_\eta$}

In what follows we adopt the notation:
$$h:=\frac{f}{g^{\ka}}.$$
We now define our measure
on $\K_\eta$ which we will eventually see satisfies (\ref{Wanttoshow}). For each level we distribute mass according to the following rules.

When $n=1$ we have that $L = B_0 := \K(1)$ and let $\mu(L):=1$. 

For balls in $K(n)$, with $n > 2$, we distribute mass inductively. Therefore, let $n \geq 2$ and suppose $\mu(B)$ is defined for each $B\in K(n-1)$. 
Let $L$ be a ball in $K(n)$. Since the balls in $K(n-1)$ are disjoint there is a unique ball $B\in K(n-1)$ satisfying $L\subset B$. By \eqref{local level}, \eqref{local level components} and \eqref{K_l notation}, we know that
$$
K(n,B) := \bigcup_{\aj\in \widetilde K_{l_B}(n,B)}\caj.
$$
Therefore, $L$ is contained in one of the collections $\cC\ajj$ appearing in the above union. We define the mass on $L$ to be
$$
\mu(L):=\frac{1}{\#\cC\ajj}\times \frac{h(\U_{j'})^{\frac{1}{1-{\ka}}}}{\sum\limits_{\aj\in \widetilde K_{l_B}(n,B)}{h(\U_{j})^{\frac{1}{1-{\ka}}}}}\times\mu(B).
$$
This quantity is well-defined in light of the preceding comment.

Proceeding inductively we see that $\mu$ is defined for each ball appearing in the construction of $\K_{\eta}$. We can extend $\mu$ uniquely in a standard way to all Borel subsets of $X$ to give a probability measure $\mu$ supported on
$\K_{\eta}$ (see, for example, \mbox{\cite[Proposition 1.7]{Falconer ref}} for further details). Given a Borel subset of $X$, say $F$, we let
\[ \mu(F):= \mu(F \cap  \K_{\eta})  \; = \; \inf\;\sum_{L\in{\cal C}(F)}\mu(L),\]
where the infimum is taken over all covers ${\cal C}(F)$ of $F
\cap \K_{\eta}$ by balls  $L \in \bigcup\limits_{n\in\N}K(n)$.

Let us conclude this section by observing that for any $L\in K(n)$ we have
\begin{align} \label{initial L measure estimate}
\mu(L)
&\leq \frac{1}{d_1\left(\frac{f(\U_{j'})}{g(\U_{j'})}\right)^{\frac{\ka}{1-\ka}}}\times \frac{h(\U_{j'})^{\frac{1}{1-\ka}}}{\sum\limits_{\aj\in \widetilde K_{l_B}(n,B)}{h(\U_{j})^{\frac{1}{1-\ka}}}}\times\mu(B) \nonumber \\[3ex]
&= \frac{f(\U_{j'})}{d_1\sum\limits_{\aj\in \widetilde K_{l_B}(n,B)}{h(\U_{j})^{\frac{1}{1-{\ka}}}}}\times\mu(B).
\end{align}
This follows from (\ref{cardinality of C(A;n)}) and the definition of $h$.

\subsection{The measure of a  ball in  the Cantor set construction} \label{Cantor construction measure estimates section}

Our ultimate goal is to prove that \eqref{Wanttoshow} is satisfied for any ball $D$ of sufficiently small radius. Moving towards that goal, we first prove that
\begin{align} \label{measure of balls in Cantor construction}
\mu(L) \ll \frac{V^f(L)}{\eta}
\end{align}
for any ball $L\in K(n)$ for $n \geq 2$.
We start with $n=2$ and then tackle higher levels of the Cantor set by induction. Let us fix a ball \mbox{$L \in K(2)=K(2,B_0)$.} Now, let $\ajj\in\widetilde K_{l_{B_0}}(2,B_0)$ be such that
$L\in \cC\ajj$. Using the upper bound given by (\ref{initial L measure estimate}), the definition of $\mu$, and the fact that $\mu(B_0)=1$, we obtain
\begin{align} \label{level 2 measure 1}
\mu(L) &\leq \frac{f(\U_{j'})}{d_1\sum\limits_{\aj\in \widetilde K_{l_{B_0}}(2,B_0)}{h(\U_{j})^{\frac{1}{1-\ka}}}}~~.
\end{align}
Using properties \textbf{(P3)} and \textbf{(P5)} of the Cantor set construction we obtain
\begin{align*}
\sum\limits_{\aj\in \widetilde K_{l_{B_0}}(2,B_0)}{h(\U_{j})^{\frac{1}{1-{\ka}}}}
&= \sum\limits_{\aj\in \widetilde K_{l_{B_0}}(2,B_0)}{V^g(A)} \\[1ex]
&=\sum_{i=1}^{l_{B_0}}{\sum_{\aj\in \widetilde K(2,B_0,i)}{V^g(A)}} \\[1ex]
&\stackrel{\textbf{(P3)}}\geq \sum_{i=1}^{l_{B_0}}{c_6V^g(B_0)} \\[1ex]
&= l_{B_0}c_6V^g(B_0) \\[1ex]
&\stackrel{\eqref{gball}}\geq l_{B_0}\frac{c_6}{c_2}\cH^g(B_0) \\[1ex]
&\stackrel{\textbf{(P5)}}\geq \frac{c_2\eta}{c_6\cH^g(B_0)}\frac{c_6}{c_2}\cH^g(B_0) ~=~ \eta.
\end{align*}
Combining this estimate with (\ref{level 2 measure 1}), and observing that $f(\U_{j'}) = V^f(L)$, we obtain~\eqref{measure of balls in Cantor construction} as required.

We now consider $n>2$. Assume that (\ref{measure of balls in Cantor construction}) holds for all balls in $K(n-1)$. Let $L$ be an arbitrary ball in $K(n)$ and let $B\in K(n-1)$ be the unique ball such that $L\in K(n,B)$. Moreover, suppose $\ajj\in\widetilde K_{l_{B}}(n,B)$ is the unique $\ajj$ such that $L\in \cC\ajj$.
By (\ref{initial L measure estimate}) and our induction hypothesis we have
\begin{align}
\mu(L) &\ll \frac{f(\U_{j'})}{d_1\sum\limits_{\aj\in \widetilde K_{l_B}(n,B)}{h(\U_{j})^{\frac{1}{1-\ka}}}}\times \frac{V^f(B)}{\eta}.\label{vb8}
\end{align}
Bounding the denominator of \eqref{vb8} we have
\begin{align}
\sum\limits_{\aj\in \widetilde K_{l_B}(n,B)}{h(\U_{j})^{\frac{1}{1-\ka}}}
&= \sum_{i=1}^{l_{B}}{\sum_{\aj\in \widetilde K(n,B,i)}{V^g(A)}}\nonumber \\[1ex]
&\stackrel{\textbf{(P3)}}\geq \sum_{i=1}^{l_{B}}{c_6V^g(B)} \nonumber \\[1ex]
&= l_Bc_6V^g(B) \nonumber \\[1ex]
&\stackrel{\textbf{(P5)}}\geq \frac{V^f(B)}{c_6V^g(B)}c_6V^g(B) \nonumber \\[1ex]
&=V^f(B).\label{vb9}
\end{align}
Combining \eqref{vb8} and \eqref{vb9} we see that (\ref{measure of balls in Cantor construction}) holds for $L$. By induction \eqref{measure of balls in Cantor construction} holds for all $L\in K(n)$ for $n\geq 2$.

\subsection{The measure of an  arbitrary ball} \label{arbitrary ball measure estimates section}

Let $r_0 := \min\{r(B): B \in K(2)\}$ and take an arbitrary ball $D$ such that $r(D) < r_0$. To conclude our proof of Theorem~\ref{general mtp theorem} it suffices to prove (\ref{Wanttoshow}) for $D$, that is we wish to show that
\[\mu(D) \ll \frac{V^f(D)}{\eta},\]
where the implied constant is independent of $D$ and $\eta$. To prove this bound we will make use of the following lemma from \cite{BV MTP}. This statement was originally (implicitly) proved in the setting of Euclidean space equipped with the usual metric. With virtually no change required to the proof, the same statement holds in an arbitrary metric space.
\begin{lemma}\label{separation lemma}
	Let $A:=B(x_A,r_A)$ and $M:=B(x_M,r_M)$ be arbitrary balls in a metric space $(X,d)$ such that $A\cap M\not=\emptyset$ and $A\setminus(cM)\not=\emptyset$ for some $c\ge3$. Then $r_M\,\le \,r_A$ and $cM\subset 5A$.
\end{lemma}

Recall that our measure $\mu$ is supported on $\K_\eta$ and we proved in the previous section that it satisfies the above inequality whenever $D$ is a ball in our Cantor set construction. Consquently, without loss of generality, we may assume that $D$ satisfies the following two properties:
\begin{itemize}
	\item{$D \cap \K_{\eta} \neq \emptyset$;}
	\item{for every $n$ large enough $D$ intersects at least two balls in $K(n)$.}
\end{itemize}

If $D \cap \K_{\eta} = \emptyset$ then $\mu(D)=0$ since $\mu$ is supported on $\K_\eta.$ If the second assumption were false then $D$ would intersect exactly one ball, say $L_{n_i}$, at level~$n_i$ for infinitely many $i \in \N$. Then, by \eqref{measure of balls in Cantor construction}, we would have \mbox{$\mu(D) \leq \mu(L_{n_i}) \to 0$} as~$i\to\infty.$ So, if either of the above two assumptions fail we have $\mu(D)=0$ and (\ref{Wanttoshow}) holds trivially.

By these two assumptions there exists a well-defined maximal integer $n$ such that
\begin{align}\label{number of balls intersected}
\text{$D$ intersects at least 2 balls from $K(n)$}
\end{align}
and
\begin{align*}
\text{$D$ intersects only one ball $B$ from $K(n-1)$}.
\end{align*}

Since $r_0 = \min\{r(B): B \in K(2)\}$ it follows that $n > 2$. Suppose \mbox{$B\in K(n-1)$} is the unique ball which has non-empty intersection with $D$, then we may also assume that $r(D) < r(B)$. To see why, suppose otherwise that $r(B) \leq r(D)$. Since $D \cap \K_\eta \subset B$ and $f$ is increasing, it would follow from (\ref{measure of balls in Cantor construction}) that
\[\mu(D) \leq \mu(B) \ll \frac{V^f(B)}{\eta} = \frac{f(r(B))}{\eta} \leq \frac{f(r(D))}{\eta} = \frac{V^f(D)}{\eta}, \]
and \eqref{Wanttoshow} would be satisfied.

Note that, since $K(n,B)$ forms a cover of $D \cap \K_{\eta}$, we have
\begin{align} \label{measure sum}
\mu(D) &\leq \sum_{i=1}^{l_B}{\sum_{L \in K(n,B,i): L \cap D \neq \emptyset}{\mu(L)}} \nonumber \\
&= \sum_{i=1}^{l_B}{\sum_{\aj \in\widetilde K(n,B,i)}~{\sum_{\substack{L \in \caj \\ L \cap D \neq \emptyset}}{\mu(L)}}} \nonumber \\
&\stackrel{(\ref{measure of balls in Cantor construction})}{\ll}~ \sum_{i=1}^{l_B}{\sum_{\aj \in \widetilde K(n,B,i)}~{\sum_{\substack{L \in \caj \\ L \cap D \neq \emptyset}}{\frac{V^f(L)}{\eta}}}}.
\end{align}

The remainder of the proof of Theorem~\ref{general mtp theorem} will be concerned with showing that \eqref{measure sum} is suitably bounded. In order to perform this task, it is useful to partition sub-levels into the following cases:

\bigskip

\noindent \underline{{\em Case 1}\/} :~ Sub-levels $K(n,B,i)$ for which
\[\#\{L \in K(n,B,i): L \cap D \neq \emptyset\} = 1.\]

\bigskip

\noindent \underline{{\em Case 2}\/} :~ Sub-levels $K(n,B,i)$ for which
\[\#\{L \in K(n,B,i): L \cap D \neq \emptyset\} \geq 2\quad\text{and}\]
\[\#\{\aj\in\widetilde K(n,B,i) \text{ with } D \cap L\neq \emptyset~\text{for some }L\in\caj\} \geq 2.\]

\bigskip

\noindent \underline{{\em Case 3}\/} :~ Sub-levels $K(n,B,i)$ for which
\[\#\{L \in K(n,B,i): L \cap D \neq \emptyset\} \geq 2\quad\text{and}\]
\[\#\{\aj\in\widetilde K(n,B,i) \text{ with } D \cap L\neq \emptyset~\text{for some }L\in\caj\} = 1.\]

Technically we should also consider those sub-levels $K(n,B,i)$ for which \mbox{$\#\{L \in K(n,B,i): L \cap D \neq \emptyset\} = 0$.} However, these sub-levels make no contribution to the sum on the right-hand side of (\ref{measure sum}) and can therefore be omitted.

\medskip

\noindent {\it Dealing with  Case 1}.
Let $K(n,B,i^*)$ be the first sub-level whose intersection with $D$ is described by Case 1. There is a unique ball $L^*$ in $K(n,B,i^*)$ satisfying $L^* \cap D \neq \emptyset$.
We know by (\ref{number of balls intersected}) that there exists another ball $M \in K(n,B)$ such that $M \cap D \neq \emptyset$. Moreover, we also know that $3L^*\cap 3M=\emptyset$ by property \textbf{(P1)}. Therefore $D \setminus 3L^* \neq \emptyset$ and so, applying Lemma~\ref{separation lemma}, we have $r(L^*) \leq r(D)$. Consequently, since $f$ is a dimension function and hence increasing,
\begin{align} \label{case 1 volume comparison}
V^f(L^*) \leq V^f(D).
\end{align}
Using property \textbf{(P4)} we know that for $i\in\{i^*+1,\ldots,l_B\}$ and $L\in
K(n,B,i)$, we have
$$
V^f(L)=f(r(L))\le 2^{-(i-i^*)}\ f(r(L^*))=2^{-(i-i^*)}\ V^f(L^*).
$$
Combining this inequality with (\ref{case 1 volume comparison}), we see that the contribution to the right-hand side of (\ref{measure sum}) from Case 1 is:
\begin{align} \label{case 1}
\sum_{i \in \text{Case 1}}~~{\sum_{\substack{L \in K(n,B,i)\\ L \cap D \neq \emptyset}}{\frac{V^f(L)}{\eta}}}
&\leq \sum_{i\ge i^*}~2^{-(i-i^*)}{\frac{V^f(L^*)}{\eta}} \le 2{\frac{V^f(L^*)}{\eta}}\le2\frac{V^f(D)}{\eta}.
\end{align}

\medskip

\noindent {\it Dealing with  Case 2}.
Let $K(n,B,i)$ be a sub-level whose intersection with~$D$ is described by Case 2. Thus, there exist distinct balls
$\aj$ and $(A';j')$ in $\widetilde K(n,B,i),$ and corresponding balls $L\in\caj$ and $L'\in\cC(A';j')$ satisfying $L \cap D \neq \emptyset$ and $L' \cap D \neq \emptyset$.
Since $L\subset A$ and $L'\subset A'$ we have $A\cap D\neq\emptyset$ and $A'\cap D\neq\emptyset$. By property~\textbf{(P2)} of our construction we know that the the balls $3A$ and $3A'$ are disjoint and contained in $B$. Therefore $D \setminus 3A \neq \emptyset$ and, applying Lemma~\ref{separation lemma}, we see that $r(A) \leq r(D)$ and $A\subset 3A \subset 5D$. By the same reasoning we also have $A'\subset 3A' \subset 5D$. Hence, on using (\ref{cardinality of C(A;n)}) we get that the contribution to the right-hand side of (\ref{measure sum}) from Case 2 is estimated as follows
\begin{align*}
\sum_{i \in \text{Case 2}}~{\sum_{\aj\in \widetilde K(n,B,i)}~{\sum_{\substack{L \in \caj \\ L \cap D \neq \emptyset}}{\frac{V^f(L)}{\eta}}}}
&\leq \sum_{i \in \text{Case 2}}~~{\sum_{\substack{\aj\in\widetilde K(n,B,i) \\ A \subset 5D}}{\#\caj\frac{f(\U_j)}{\eta}}} \\[2ex]
&\stackrel{(\ref{cardinality of C(A;n)})}\ll \sum_{i \in \text{Case 2}}~{\sum_{\substack{\aj\in\widetilde K(n,B,i) \\ A \subset 5D}}{\left(\frac{f(\U_j)}{g(\U_j)}\right)^\frac{\ka}{1-\ka}\frac{f(\U_j)}{\eta}}} \\[2ex]
&= \sum_{i \in \text{Case 2}}~{\sum_{\substack{\aj\in\widetilde K(n,B,i) \\ A \subset 5D}}{\frac{f(\U_j)^{\frac{\ka}{1-\ka}+1}}{g(\U_j)^{\frac{\ka}{1-\ka}}}\times\frac{1}{\eta}}} \\[2ex]
&= \sum_{i \in \text{Case 2}}~{\sum_{\substack{\aj\in\widetilde K(n,B,i) \\ A \subset 5D}}{\frac{f(\U_j)^{\frac{1}{1-\ka}}}{g(\U_j)^{\frac{\ka}{1-\ka}}}\times\frac{1}{\eta}}} \\[2ex]
&= \sum_{i \in \text{Case 2}}~{\sum_{\substack{\aj\in\widetilde K(n,B,i) \\ A \subset 5D}}{\frac{V^g(A)}{\eta}}}.
\end{align*}

It follows upon combining this estimate with \eqref{doubling}, \eqref{gball}, and the disjointness of balls in $\widetilde{K}(n,B,i)$ guaranteed by property \textbf{(P2)},  that
\begin{align*}
\sum_{i \in \text{Case 2}}~{\sum_{\aj\in \widetilde K(n,B,i)}~{\sum_{\substack{L \in \caj \\ L \cap D \neq \emptyset}}{\frac{V^f(L)}{\eta}}}}
&\stackrel{\eqref{gball}}\ll \frac{1}{\eta}\sum_{i \in \text{Case 2}}~{\sum_{\substack{\aj\in\widetilde K(n,B,i) \\ A\subset 5D}}{\cH^g(A)}} \\[2ex]
&\stackrel{\textbf{(P2)}}= \frac{1}{\eta}\sum_{i \in \text{Case 2}}~{\cH^g\left(\bigcup_{\substack{\aj\in\widetilde K(n,B,i) \\ A\subset 5D}}A\right)} \\[2ex]
&\leq \frac{1}{\eta}\sum_{i \in \text{Case 2}}{\cH^g(5D)} \nonumber \\[2ex]
&\stackrel{\eqref{gball}}\ll \frac{1}{\eta}\sum_{i \in \text{Case 2}}{V^g(5D)} \\[2ex]
&\stackrel{\eqref{doubling}}\ll \frac{1}{\eta}\sum_{i \in \text{Case 2}}{V^g(D)} \\[2ex]
&\leq \frac{1}{\eta}l_BV^g(D).
\end{align*}
Finally, it follows from the above estimate together with property {\bf(P5)} that
\begin{align} \label{case 2}
\sum_{i \in \text{Case 2}}~{\sum_{\aj\in \widetilde K(n,B,i)}~{\sum_{\substack{L \in \caj \\ L \cap D \neq \emptyset}}{\frac{V^f(L)}{\eta}}}}
&\stackrel{\textbf{(P5)}}\ll \frac{1}{\eta}\left(\frac{V^f(B)}{V^g(B)}\right)V^g(D) \nonumber \\[2ex]
&\ll \frac{1}{\eta}\frac{V^f(D)}{V^k(D)}V^k(D) \nonumber \\[2ex]
&= \frac{V^f(D)}{\eta}.
\end{align}To deduce the penultimate inequality we used the facts that $f/g$ is decreasing and $r(D) < r(B)$. 

\noindent {\it Dealing with  Case 3}.
For each sub-level $i$ whose intersection with $D$ is described by Case~3 there exists a unique \mbox{$(A_i;j_i)\in\widetilde K(n,B,i)$} such that $D$ has non-empty intersection with balls in $\cC(A_i;j_i)$. Let $K(n,B,i^{**})$ denote the first sub-level described by Case~3. There exists a ball $L^{**}$ in $K(n,B,i^{**})$ such that $L^{**} \cap D \neq \emptyset$.
By the assumption in (\ref{number of balls intersected}) there must exist another ball $M \in K(n,B)$ such that $M \cap D \neq \emptyset$. It follows from property \textbf{(P1)} that $3L^{**}$ and $3M$ are disjoint and so $D \setminus 3L^{**} \neq \emptyset.$ Applying Lemma~\ref{separation lemma}, we have that $r(L^{**}) \leq r(D).$ As $h$ is a dimension function it follows that
\begin{align} \label{case 1 volume comparison+}
h(r(L^{**})) \leq h(r(D)).
\end{align}
By property \textbf{(P4)} we know that for any $i\in\{i^{**}+1,\ldots,l_B\}$ and \mbox{$L\in K(n,B,i)$} we have 
\begin{equation}
\label{halfing bound}
h(r(L))\le 2^{-(i-i^{**})}\ h(r(L^{**})).
\end{equation} 
Recall that, by Lemma~\ref{C(A;n) lemma}, each $L\in \cC(A_i;j_i)$ is centred on $F_{j_i}.$ Combining this fact with the LSP, the relations given by \eqref{gball}, and the fact that the elements of $\cC(A_i;j_i)$ are disjoint, by straightforward measure theoretic considerations we have the following estimate 
\begin{equation}
\label{finalbound}\#\{L \in \cC(A_i;j_i):  L \cap D \neq \emptyset\}\ll \frac{g(\U_{j_i})^{1-\ka}g(r(D))^{\ka}}{g(\U_{j_i})}.
\end{equation}
Therefore, the contribution to the right-hand side of (\ref{measure sum}) from Case 3 can be bounded above as follows:
\begin{align}
\sum_{i \in \text{Case 3}}~{\sum_{\aj\in\widetilde K(n,B,i)}~{\sum_{\substack{L \in \caj \\ L \cap D \neq \emptyset}}{\frac{V^f(L)}{\eta}}}}
&\leq \sum_{i \in \text{Case 3}}~{\sum_{\substack{L \in \cC(A_i;j_i) \\ L \cap D \neq \emptyset}}{\frac{V^f(L)}{\eta}}} \nonumber \\[2ex]
&= \sum_{i \in \text{Case 3}}~{\sum_{\substack{L \in \cC(A_i;j_i) \\ L \cap D \neq \emptyset}}{\frac{f(\U_{j_i})}{\eta}}} \nonumber \\[2ex]
& \stackrel{\eqref{finalbound}}\ll \sum_{i \in \text{Case 3}}~{\frac{g(\U_{j_i})^{1-\ka}g(r(D))^{\ka}}{g(\U_{j_i})}\times\frac{f(\U_{j_i})}{\eta}} \nonumber \\[2ex]
&= \frac{g(r(D))^{\ka}}{\eta}\sum_{i \in \text{Case 3}}{\frac{f(\U_{j_i})}{g(\U_{j_i})^{\ka}}}\nonumber\\[2ex]
&= \frac{g(r(D))^{\ka}}{\eta}\sum_{i \in \text{Case 3}}{h(\U_{j_i})} \nonumber \\[2ex]
&\stackrel{\eqref{halfing bound}}\ll \frac{g(r(D))^{\ka}}{\eta}\sum_{i \geq i^{**}}{\frac{h(\U_{j_{i^{**}}})}{2^{i-i^{**}}}} \nonumber \\[2ex]
&\leq 2\frac{g(r(D))^{\ka}}{\eta}h(\U_{j_{i^{**}}}). \nonumber
\end{align}
Recalling (\ref{case 1 volume comparison+}) and noting that $\U_{j_i} = r(L^{**})$, we observe that
\begin{align}
\sum_{i \in \text{Case 3}}~{\sum_{\aj\in\widetilde K(n,B,i)}~{\sum_{\substack{L \in \caj \\ L \cap D \neq \emptyset}}{\frac{V^f(L)}{\eta}}}} &\ll 2\frac{g(r(D))^{\ka}}{\eta}h(r(D)) \nonumber \\
                                                         &= 2\frac{f(r(D))}{\eta} \ll \frac{V^f(D)}{\eta}\label{case 3}.
\end{align}
Combining estimates (\ref{case 1}), (\ref{case 2}) and (\ref{case 3}) with (\ref{measure sum}) gives $\mu(D) \ll \frac{V^f(D)}{\eta}$, thus proving \eqref{Wanttoshow} as desired. This completes the proof of Theorem~\ref{general mtp theorem}.

\section{An application of Theorem \ref{general mtp theorem}: Random $\limsup$ sets} \label{Application} 

In this section we give an application of Theorem \ref{general mtp theorem} to the study of random $\limsup$ sets, which is classical topic of interest within Probability Theory. We refer the reader to \cite{Jar} and the references therein for more on this problem. 

We start by imposing the additional assumptions that $(X,d)$ is a compact metric space and $g(r):=r^s$ for some $s> 0$. By rescaling if necessary, we may assume without loss of generality that $\mathcal{H}^s(X)=1.$ Let $(F_j)_{j \in \N}$ be a sequence of sets satisfying the LSP with respect to some $0\leq \kappa<1$. Assume that we are given a set of isometries $\Phi$ that are chosen randomly according to some law $\mathbb{P}$. Given a sequence of non-negative real numbers $\U:=(\U_j)_{j \in \N}$ and a randomly chosen sequence $(\phi_j)_{j \in \N} \in\Phi^\mathbb{N},$ we define the corresponding random $\limsup$ set as follows:
$$\Lambda((\phi_j),\U):=\{x\in \Delta(\phi_j(F_j),\U_j) \textrm{ for infinitely many } j\in\mathbb{N}\}.$$

We are interested in determining the $\mathbb{P}$-almost sure Hausdorff dimension and Hausdorff measure of $\Lambda((\phi_j),\U).$ When $(F_j)_{j \in \N}$ is a sequence of points, then the $\mathbb{P}$-almost sure Hausdorff dimension and Hausdorff measure of $\Lambda((\phi_j),\U)$ is well understood. When $(F_j)_{j \in \N}$ is a more exotic sequence of sets, the problem of determining the $\mathbb{P}$-almost sure metric properties of $\Lambda((\phi_j),\U)$ is more difficult. That being said, in the Euclidean setting a comprehensive description of the \mbox{$\mathbb{P}$-almost} sure metric properties of $\Lambda((\phi_j),\U)$ is given in~\cite{FengJar}. See also~\cite{JarJar} and~\cite{Per}. Our application below holds in the more general metric space setting and also provides an alternative proof for some of the important results appearing in~\cite{FengJar}. In what follows we assume that $\mathbb{P}$ satisfies the following properties: 
\begin{itemize} 
\item For any $j\in\mathbb{N}$ and $\delta > 0$ sufficiently small: 
\begin{align} \label{P property 1}
\mathbb{P}(x\in \Delta(\phi_j(F_j),\delta))=\cH^s(\Delta(F_j,\delta)).
\end{align}
\item For any sequence $(\delta_j)_{j \in \N}$ of sufficiently small numbers, we have that the sequence of events $(E_j)_{j=1}^{\infty}=(\{x\in \Delta(\phi_j(F_j),\delta_j)\})_{j=1}^{\infty}$ are independent, i.e. for any finite set $S\subset \mathbb{N}$ we have $$\mathbb{P}\left(\bigcap_{j\in S}E_j\right)=\prod_{j\in S}\mathbb{P}(E_j).$$
\end{itemize}

The main result of this section is Theorem \ref{Random application} below. For convenience it is stated for $(\U_j)_{j \in \N}$ of the form $(j^{-\tau})_{j \in \N}$, although it holds in greater generality. To prove this theorem we will make use of the following well known result from Probability Theory known as the second Borel--Cantelli lemma, see \cite{Bil}.

\begin{lemma}[Second Borel--Cantelli Lemma]
	\label{BC lemma}
	Let $(X,\mathcal{A},\mu)$ be a probability space. Let $(E_j)_{j \in \N}$ be an independent sequence of events such that $\sum_{j=1}^{\infty}\mu(E_j)=\infty,$ then $\mu(\limsup_{j \to \infty} E_j)=1.$
\end{lemma}

\begin{theorem} \label{Random application}
Let $(X,d),$ $\Phi,$ and $\mathbb{P}$ be as above. Suppose $(F_j)_{j \in \N}$ satisfies the LSP with respect to some $0\leq \kappa<1$ and let $\tau>\frac{1}{s-\kappa s}$. Then, for $\mathbb{P}$-almost every $(\phi_j)_{j \in \N} \in \Phi^{\mathbb{N}}$, we have 
\[\dimh(\Lambda((\phi_j),(j^{-\tau})))=\kappa s+\frac{1}{\tau} \quad \text{and} \quad \cH^{\kappa s+1/\tau}(\Lambda((\phi_j),(j^{-\tau})))=\infty.\]
\end{theorem}

\begin{proof}
It follows from the LSP that for any $j \in \N$ and $\delta>0$ sufficiently small we have 
\begin{equation}
\label{global scaling}
\cH^s(\Delta(F_j,\delta)) \asymp \delta^{s-\kappa s}.
\end{equation} Using the notation of Section \ref{KGB Section}, we see that when $f(r)=r^t$ we have \mbox{$\tU_j=j^{\frac{\tau(\kappa s-t)}{s-\kappa s}}$.} When $t=\kappa s+1/\tau$ we have 
$$\sum_{j=1}^{\infty}\mathbb{P}(x\in \Delta(\phi_j(F_j),\tU_j))\stackrel{\eqref{P property 1}}=\sum_{j=1}^{\infty}\cH^s(\Delta(F_j,\tU_j))\stackrel{\eqref{global scaling}}\asymp \sum_{j=1}^{\infty}j^{-1}=\infty.$$
Applying Lemma \ref{BC lemma}, we may assert that for a fixed $x\in X$ we have $$\mathbb{P}(x\in \Lambda((\phi_j),\tU))=1.$$ Applying Fubini's Theorem, we obtain that for $\mathbb{P}$-almost every $(\phi_j)_{j \in \N}\in\Phi^{\mathbb{N}}$ we have 
\begin{equation}
\label{Fubini}
\cH^s(\Lambda((\phi_j),\tU))=1,
\end{equation} i.e. our random $\limsup$ set has full measure. Equation \eqref{Fubini} tells us that for almost every $(\phi_j)_{j \in \N}\in\Phi^{\mathbb{N}}$ the assumptions of Theorem~\ref{general mtp theorem} are satisfied. Therefore, we may apply Theorem~\ref{general mtp theorem} and conclude that for almost every $(\phi_j)_{j \in \N}\in\Phi^{\mathbb{N}}$ we have
$$\cH^{\kappa s+1/\tau}(\Lambda((\phi_j),(j^{-\tau})))=\infty,$$
since we assumed that $\cH^s(X)=1$ and $\ka s + 1/\tau < s$. 

To complete our proof, it suffices to show that $\cH^{t}(\Lambda((\phi_j),(j^{-\tau})))=0$ for any $(\phi_j)_{j \in \N}\in\Phi^{\mathbb{N}}$ and $t>\kappa s + 1/\tau.$ Therefore fix $(\phi_j)_{j \in \N}\in\Phi^{\mathbb{N}}$. Applying the $5r$-covering lemma (Lemma~\ref{5r}) we may construct for each $j\in\mathbb{N}$ a finite disjoint collection of balls $\{B(x_{l,j},j^{-\tau})\}_{l\in Y_j}$ such that $\{B(x_{l,j},5j^{-\tau})\}_{l\in Y_j}$ covers $\Delta(\phi(F_j),j^{-\tau})$. Notice that the collection $\{B(x_{l,j},5j^{-\tau})\}_{l\in Y_j,j\geq N}$ covers $\Lambda((\phi_j),(j^{-\tau}))$ for any $N\in\mathbb{N}$. Furthermore, combining \eqref{gball} and \eqref{global scaling} we see that 
\begin{equation}
\label{count}\# Y_j \ll \frac{\mathcal{H}^s(\Delta(\phi_{j}(F_j),j^{-\tau}))}{j^{-\tau s}}\asymp \frac{j^{-\tau(s-\kappa s)}}{j^{-\tau s}}=j^{\tau\kappa s}.
\end{equation}
 Taking $\rho>0$, $\varepsilon>0$ and $t>\kappa s+1/\tau,$ we may choose $N$ sufficiently large such that 
$$\cH^t_{\rho}(\Lambda((\phi_j),(j^{-\tau})))\leq \sum_{j=N}^{\infty}\sum_{l\in Y_j}(5j^{-\tau})^t \stackrel{\eqref{count}}\ll 5^t \sum_{j=N}^{\infty}j^{\tau \kappa s-\tau t}<\varepsilon.$$
Since $\varepsilon$ and $\rho$ were arbitrary we may deduce that $\cH^{t}(\Lambda((\phi_j),(j^{-\tau})))=0$ as required. Therefore, for $\mathbb{P}$-almost every $(\phi_j)_{j \in \N} \in \Phi^{\N}$, we have 
\[\dimh(\La((\phi_j),(j^{-\tau}))) = \ka s + \frac{1}{\tau}.\qedhere\]
\end{proof}

\noindent {\bf Acknowledgements.}
This work began while both authors were in attendance at the program on \emph{Fractal Geometry and Dynamics} at the Mittag--Leffler Institut in November 2017. We are grateful both to the organisers of the program and the staff at the Institut for making it such a wonderful environment for doing mathematics. DA is grateful to Victor Beresnevich and Sanju Velani, her PhD supervisors, for introducing her to the Mass Transference Principle and the topic of Metric Diophantine Approximation in general.

%

%

\Addresses

\end{document}